\documentclass{amsart}

\usepackage{enumerate, amsmath, amsfonts, amssymb, amsthm, wasysym, graphics, graphicx, xcolor, url, hyperref, hypcap, frcursive}
\hypersetup{colorlinks=true, citecolor=darkblue, linkcolor=darkblue}
\usepackage[all]{xy}
\usepackage{tikz}
\graphicspath{{figures/}}

\title[EL-labelings and canonical spanning trees for subword complexes]{EL-labelings and canonical spanning trees \\ for subword complexes}

\thanks{V.\,P.~was supported by the spanish MICINN grant MTM2011-22792, by the french ANR grant EGOS 12 JS02 002 01, by the European Research Project ExploreMaps (ERC StG 208471), and by a postdoctoral grant of the Fields Institute of Toronto.}

\author{Vincent Pilaud}
\address{CNRS \& LIX, \'Ecole Polytechnique, Palaiseau}
\email{vincent.pilaud@lix.polytechnique.fr}
\urladdr{http://www.lix.polytechnique.fr/~pilaud/}

\author{Christian Stump}
\address{Institut f\"ur Algebra, Zahlentheorie, Diskrete Mathematik, Universit\"at Hannover}
\email{stump@math.uni-hannover.de}
\urladdr{http://homepage.univie.ac.at/christian.stump/}

\newtheorem{theorem}{Theorem}[section]
\newtheorem{corollary}[theorem]{Corollary}
\newtheorem{proposition}[theorem]{Proposition}
\newtheorem{lemma}[theorem]{Lemma}

\theoremstyle{definition}
\newtheorem{example}[theorem]{Example}
\newtheorem{remark}[theorem]{Remark}

\newcommand{\R}{\mathbb{R}} 
\newcommand{\N}{\mathbb{N}} 
\newcommand{\Z}{\mathbb{Z}} 
\newcommand{\cX}{\mathcal{X}} 
\newcommand{\cN}{\mathcal{N}} 
\newcommand{\cT}{\mathcal{T}} 
\newcommand{\fS}{\mathfrak{S}} 

\newcommand{\set}[2]{\left\{ #1 \;\middle|\; #2 \right\}} 
\newcommand{\multiset}[2]{\left\{\!\!\left\{ #1 \;\middle|\; #2 \right\}\!\!\right\}} 
\newcommand{\ssm}{\smallsetminus} 
\newcommand{\dotprod}[2]{\langle #1 | #2 \rangle} 
\newcommand{\symdif}{\triangle} 
\newcommand{\zero}{\mathbf{0}} 
\newcommand{\one}{\mathbf{1}} 
\newcommand{\eqdef}{\mbox{\,\raisebox{0.2ex}{\scriptsize\ensuremath{\mathrm:}}\ensuremath{=}\,}} 

\newcommand{\sq}[1]{{\rm #1}} 
\newcommand{\Q}{\sq{Q}} 
\newcommand{\q}{\sq{q}} 
\newcommand{\eraseFirst}{\vdash} 
\newcommand{\eraseLast}{\dashv} 
\newcommand{\shiftRight}[1]{#1^{\rightarrow}} 
\newcommand{\shiftLeft}[1]{#1^{\leftarrow}} 

\newcommand{\wordprod}[2]{\Pi{#1}_{#2}} 
\newcommand{\length}{\ell} 
\newcommand{\subwordComplex}[1][\Q,\rho]{\mathcal{SC}(#1)} 
\newcommand{\Root}[2]{\mathsf{r}(#1,#2)} 
\newcommand{\Roots}[1]{\mathsf{R}(#1)} 
\newcommand{\flipGraph}[1][\Q,\rho]{\mathcal{G}(#1)} 
\newcommand{\flipPoset}[1][\Q,\rho]{\Gamma(#1)} 
\newcommand{\facets}[1][\Q,\rho]{\mathcal{F}(#1)} 

\newcommand{\positiveEdgeLabel}{\mathsf{p}} 
\newcommand{\negativeEdgeLabel}{\mathsf{n}} 
\newcommand{\positiveFacet}[1][\Q,\rho]{\mathsf{P}(#1)} 
\newcommand{\negativeFacet}[1][\Q,\rho]{\mathsf{N}(#1)} 
\newcommand{\positiveSourceTree}[1][\Q,\rho]{\mathcal{P}(#1)} 
\newcommand{\negativeSourceTree}[1][\Q,\rho]{\mathcal{N}(#1)} 
\newcommand{\positiveSinkTree}[1][\Q,\rho]{\mathcal{P}^*(#1)} 
\newcommand{\negativeSinkTree}[1][\Q,\rho]{\mathcal{N}^*(#1)} 
\newcommand{\flipsOut}{\mathsf{P}} 
\newcommand{\flipsIn}{\mathsf{N}} 

\newcommand{\sortable}[2]{\operatorname{\textsc{Sort}}_{#1}(#2)} 
\newcommand{\cwo}[1]{\sw{w_\circ}{#1}} 
\newcommand{\sortingTree}[1][\sq{c}]{\mathcal{T}(#1)} 
\newcommand{\sw}[2]{\sq{#1}(\sq{#2})} 

\newcommand{\Qdup}{\Q^{\sq{dup}}} 
\newcommand{\Qex}{\Q^{\sq{ex}}} 
\newcommand{\rhoex}{\rho^{\sq{ex}}} 
\newcommand{\Iex}{I^{\sq{ex}}} 
\newcommand{\Jex}{J^{\sq{ex}}} 

\newcommand{\even}{\mathrm{even}} 
\newcommand{\odd}{\mathrm{odd}} 
\newcommand{\edge}{\mathbin{\tikz [semithick, baseline=-0.2ex,-latex, ->] \draw [->] (0pt,0.4ex) -- (0.7em,0.4ex);}} 
\newcommand{\edgePositiveLabel}[1]{\mathbin{\tikz [semithick, baseline=-0.2ex,-latex, ->] \draw [-] (0pt,0.4ex) -- (0.7em,0.4ex); #1 \tikz [semithick, baseline=-0.2ex,-latex, ->] \draw [->] (0pt,0.4ex) -- (0.7em,0.4ex);}} 
\newcommand{\directedPath}{\mathbin{\tikz [semithick, baseline=-0.2ex,-latex, ->, densely dashed] \draw [->,densely dashed] (0pt,0.4ex) -- (1.0em,0.4ex);}} 
\newcommand{\sep}{\edge} 
\newcommand{\sourceTree}{\mathcal{T}} 
\newcommand{\sinkTree}{\mathcal{T}^*} 

\DeclareMathOperator{\conv}{conv} 
\DeclareMathOperator{\inv}{inv} 
\DeclareMathOperator{\vect}{vect} 
\DeclareMathOperator{\join}{\star} 

\usepackage{todonotes}

\newcommand{\fref}[1]{Figure~\ref{#1}} 
\newcommand{\ie}{\textit{i.e.}~} 
\newcommand{\eg}{\textit{e.g.}~} 
\newcommand{\viceversa}{\textit{vice versa}} 
\newcommand{\ordinal}{\textsuperscript{th}} 
\definecolor{darkblue}{rgb}{0,0,0.7} 
\newcommand{\darkblue}{\color{darkblue}} 
\newcommand{\defn}[1]{\emph{\darkblue #1}} 
\renewcommand{\paragraph}[1]{\bigskip\noindent\textbf{#1.}} 

\begin{document}

\begin{abstract}
We describe edge labelings of the increasing flip graph of a subword complex on a finite Coxeter group, and study applications thereof.
On the one hand, we show that they provide canonical spanning trees of the facet-ridge graph of the subword complex, describe inductively these trees, and present their close relations to greedy facets.
Searching these trees yields an efficient algorithm to generate all facets of the subword complex, which extends the greedy flip algorithm for pointed pseudotriangulations.
On the other hand, when the increasing flip graph is a Hasse diagram, we show that the edge labeling is indeed an EL-labeling and derive further combinatorial properties of paths in the increasing flip graph.
These results apply in particular to Cambrian lattices, in which case a similar EL-labeling was recently studied by M.~Kallipoliti and H.~M\"uhle.
\end{abstract}

\maketitle

\tableofcontents

\section{Introduction}

Subword complexes on Coxeter groups were defined and studied by A.~Knutson and E.~Miller in the context of Gr\"obner geometry of Schubert varieties~\cite{KnutsonMiller-subwordComplex,KnutsonMiller-GroebnerGeometry}.
Type~$A$ spherical subword complexes can be visually interpreted using pseudoline arrangements on primitive sorting networks.
These were studied by V.~Pilaud and M.~Pocchiola \cite{PilaudPocchiola} as combinatorial models for pointed pseudotriangulations of planar point sets~\cite{RoteSantosStreinu-survey} and for multitriangulations of convex polygons~\cite{PilaudSantos-multitriangulations}.
These two families of geometric graphs extend in two different ways the family of triangulations of a convex polygon.

The greedy flip algorithm was initially designed to generate all pointed pseudotriangulations of a given set of points or convex bodies in general position in the plane~\cite{PocchiolaVegter, BronnimannKettnerPocchiolaSnoeying}.
It was then extended in~\cite{PilaudPocchiola} to generate all pseudoline arrangements supported by a given primitive sorting network.
The key step in this algorithm is to construct a spanning tree of the flip graph on the combinatorial objects, which has to be sufficiently canonical to be visited in polynomial time per node and polynomial working space.

In the present paper, we study natural edge lexicographic labelings of the increasing flip graph of a subword complex on any finite Coxeter group.
As a first line of applications of these EL-labelings, we obtain canonical spanning trees of the flip graph of any subword complex.
We provide alternative descriptions of these trees based on their close relations to greedy facets, which are defined and studied in this paper.
Moreover, searching these trees provides an efficient algorithm to generate all facets of the subword complex.
For type~$A$ spherical subword complexes, the resulting algorithm is that of~\cite{PilaudPocchiola}, although the presentation is quite different.

The second line of applications of the EL-labelings concerns combinatorial properties ensuing from EL-shellability~\cite{Bjorner, BjornerWachs3}.
Indeed, when the increasing flip graph is the Hasse diagram of the increasing flip poset, this poset is EL-shellable, and we can compute its M\"obius function.
These results extend recent work of M.~Kallipoliti and H.~M\"uhle~\cite{KallipolitiMuhle} on EL-shellability of N.~Reading's Cambrian lattices~\cite{Reading-latticeCongruences, Reading-cambrianLattices, Reading-coxeterSortable, Reading-sortableElements}, which are, for finite Coxeter groups, increasing flip posets of specific subword complexes studied by C.~Ceballos, J.-P.~Labb\'e and C.~Stump~\cite{CeballosLabbeStump} and by the authors in~\cite{PilaudStump}.


\section{Edge labelings of graphs and posets}
\label{sec:ELlabelings}

In~\cite{Bjorner}, A.~Bj\"orner introduced EL-labelings of partially ordered sets to study topological properties of their order complexes.
These labelings are edge labelings of the Hasse diagrams of the posets with certain combinatorial properties.
In this paper, we consider edge labelings of finite, acyclic, directed graphs which might differ from the Hasse diagrams of their transitive closures.


\subsection{ER-labelings of graphs and associated spanning trees}
\label{subsec:ERlabelings}

Let~$G \eqdef (V,E)$ be a finite, acyclic, directed graph.
For~$u,v \in V$, we write~$u \edge v$ if there is an edge from~$u$ to~$v$ in~$G$, and $u \directedPath v$ if there is a \defn{path} ${u = x_1 \edge x_2 \edge \cdots \edge x_{\ell+1} = v}$ from~$u$ to~$v$ in~$G$ (this path has \defn{length}~$\ell$).
The \defn{interval}~$[u,v]$ in~$G$ is the set of vertices~$w \in V$ such that~$u \directedPath w \directedPath v$.

An \defn{edge labeling} of~$G$ is a map~$\lambda : E \to \N$.
It induces a labeling~$\lambda(p)$ of any path~$p : x_1 \edge x_2 \edge \cdots \edge x_\ell \edge x_{\ell+1}$ given by $\lambda(p) \eqdef \lambda(x_1 \sep x_2) \cdots \lambda(x_\ell \sep x_{\ell+1})$.
The path~$p$ is \defn{$\lambda$-rising} (resp.~\defn{$\lambda$-falling}) if~$\lambda(p)$ is strictly increasing (resp. weakly decreasing).
The labeling~$\lambda$ is an \defn{edge rising labeling} of~$G$ (or \defn{ER-labeling} for short) if there is a unique $\lambda$-rising path~$p$ between any vertices~$u,v \in V$ with~$u \directedPath v$.

\begin{remark}[Spanning trees]
\label{rem:spanningTrees}
Let~$u,v \in V$, and~$\lambda : E \to \N$ be an ER-labeling of~$G$.
Then the union of all $\lambda$-rising paths from~$u$ to any other vertex of the interval~$[u,v]$ forms a spanning tree of~$[u,v]$, rooted at and directed away from~$u$.
We call it the \defn{$\lambda$-source tree} of~$[u,v]$ and denote it by~$\sourceTree_\lambda([u,v])$.
Similarly, the union of all $\lambda$-rising paths from any vertex of the interval~$[u,v]$ to~$v$ forms a spanning tree of~$[u,v]$, rooted at and directed towards~$v$.
We call it the \defn{$\lambda$-sink tree} of~$[u,v]$ and denote it by~$\sinkTree_\lambda([u,v])$.
In particular, if~$G$ has a unique source and a unique sink, this provides two canonical spanning trees~$\sourceTree_\lambda(G)$ and~$\sinkTree_\lambda(G)$ for the graph~$G$ itself.
\end{remark}

\begin{example}[Cube]
\label{exm:cube}
Consider the $1$-skeleton~$\square_d$ of the $d$-dimensional cube~$[0,1]^d$, directed from~$\zero \eqdef (0,\dots,0)$ to~$\one \eqdef (1,\dots,1)$.
Its vertices are the elements of~$\{0,1\}^d$ and its edges are the pairs of vertices which differ in a unique position.
Note that~$\varepsilon \eqdef (\varepsilon_1, \dots, \varepsilon_d) \directedPath \varepsilon' \eqdef (\varepsilon'_1, \dots, \varepsilon'_d)$ if and only if~$\varepsilon_k \le \varepsilon'_k$ for all~$k \in [d]$.

For any edge $\varepsilon \edge \varepsilon'$ of~$\square_d$, let~$\lambda(\varepsilon \edge \varepsilon')$ denote the unique position in~$[d]$ where~$\varepsilon$ and~$\varepsilon'$ differ.
Then the map~$\lambda$ is an ER-labeling of~$\square_d$.
If~$\varepsilon \in \{0,1\}^d \ssm \zero$, then the father of~$\varepsilon$ in~$\sourceTree_\lambda(\square_d)$ is obtained from~$\varepsilon$ by changing its last~$1$ into a~$0$.
Similarly, if~${\varepsilon \in \{0,1\}^d \ssm \one}$, then the father of~$\varepsilon$ in~$\sinkTree_\lambda(\square_d)$ is obtained from~$\varepsilon$ by changing its first~$0$ into a~$1$.
See \fref{fig:cube}.

\begin{figure}[ht]
  \centerline{\includegraphics[width=\textwidth]{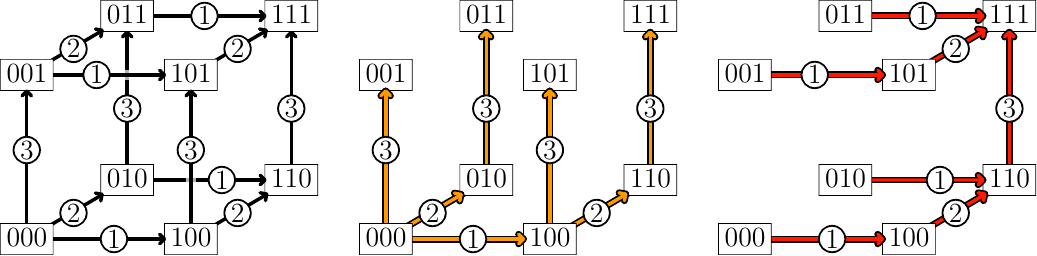}}
  \caption{An ER-labeling~$\lambda$ of the $1$-skeleton~$\square_3$ of the $3$-cube, the $\lambda$-source tree~$\sourceTree_\lambda(\square_3)$ and the $\lambda$-sink tree~$\sinkTree_\lambda(\square_3)$.}
  \label{fig:cube}
\end{figure}
\end{example}

\subsection{EL-labelings of graphs and posets}
\label{subsec:ELlabelings}

Although ER-labelings of graphs are sufficient to produce canonical spanning trees, we need the following extension for further properties.
The labeling~$\lambda : E \to \N$ is an \defn{edge lexicographic labeling} of~$G$ (or \defn{EL-labeling} for short) if for any vertices~$u,v \in V$ with~$u \directedPath v$,
\begin{enumerate}[(i)]
\item there is a unique $\lambda$-rising path~$p$ from~$u$ to~$v$, and 
\item its labeling~$\lambda(p)$ is lexicographically first among the labelings~$\lambda(p')$ of all paths~$p'$ from~$u$ to~$v$.
\end{enumerate}
For example, the ER-labeling of the $1$-skeleton of the cube presented in Example~\ref{exm:cube} is in fact an EL-labeling.

\enlargethispage{.2cm}
Remember now that one can associate a finite poset to a finite acyclic directed graph and \viceversa.
Namely,
\begin{enumerate}[(i)]
\item the \defn{transitive closure} of a finite acyclic directed graph~$G = (V,E)$ is the finite poset~${(V,\directedPath)}$;
\item the \defn{Hasse diagram} of a finite poset~$P$ is the finite acyclic directed graph whose vertices are the elements of~$P$ and whose edges are the \defn{cover relations} in $P$, \ie $u \edge v$ if~$u <_P v$ and there is no~$w \in P$ such that~$u <_P w <_P v$.
\end{enumerate}
The transitive closure of the Hasse diagram of~$P$ always coincides with~$P$, but the Hasse diagram of the transitive closure of~$G$ might also be only a subgraph of~$G$.
An \defn{EL-labeling} of the poset~$P$ is an EL-labeling of the Hasse diagram of~$P$.
If such a labeling exists, then the poset is called \defn{EL-shellable}.

As already mentioned, A.~Bj\"orner~\cite{Bjorner} originally introduced EL-labelings of finite posets to study topological properties of their order complex.
In particular, they provide a tool to compute the M\"obius function of the poset.
Recall that the \defn{M\"obius function} of the poset~$P$ is the map~$\mu : P \times P \to \Z$ defined recursively~by
$$\mu(u,v) \eqdef \begin{cases} 1 & \text{if } u = v, \\ -\sum_{u \le_P w <_P v} \mu(u,w) & \text{if } u <_P v, \\ 0 & \text{otherwise.} \end{cases}$$
When the poset is EL-shellable, this function can be computed as follows.

\begin{proposition}[\protect{\cite[Proposition~5.7]{BjornerWachs3}}]
\label{prop:Moebius}
Let~$\lambda$ be an EL-labeling of the poset~$P$.
For every~$u,v \in P$ with~$u \le_P v$, we have
$$\mu(u,v) = \even_\lambda(u,v) - \odd_\lambda(u,v),$$
where $\even_\lambda(u,v)$ (resp.~$\odd_\lambda(u,v)$) denotes the number of even (resp.~odd) length $\lambda$-falling paths from~$u$ to~$v$ in the Hasse diagram of~$P$.
\end{proposition}

\begin{example}[Cube]
\label{exm:cube2}
The directed $1$-skeleton~$\square_d$ of the $d$-dimensional cube~$[0,1]^d$ is the Hasse diagram of the boolean poset.
The edge labeling~$\lambda$ of~$\square_d$ of Example~\ref{exm:cube} is thus an EL-labeling of the boolean poset.
Moreover, for any two vertices~${\varepsilon \directedPath \varepsilon'}$ of~$\square_d$, there is a unique $\lambda$-falling path between~$\varepsilon$ and~$\varepsilon'$, whose length is the \defn{Hamming distance}~$\delta(\varepsilon,\varepsilon') \eqdef |\set{k \in [d]}{\varepsilon_k \ne \varepsilon'_k}|$.
The M\"obius function~is thus given by~$\mu(\varepsilon,\varepsilon') = (-1)^{\delta(\varepsilon,\varepsilon')}$.
In particular,~${\mu(\zero, \one) = (-1)^d}$.
\end{example}


\section{Subword complexes on Coxeter groups}
\label{sec:subwordComplexes}

\subsection{Coxeter systems}
\label{subsec:CoxeterSystems}

We recall some basic notions on Coxeter systems needed in this paper.
More background material can be found in~\cite{Humphreys}.

Let~$V$ be an $n$-dimensional Euclidean vector space.
For~$v \in V \ssm 0$, we denote by~$s_v$ the reflection interchanging~$v$ and~$-v$ while fixing pointwise the orthogonal hyperplane.
We consider a \defn{finite Coxeter group}~$W$ acting on~$V$, \ie a finite group generated by reflections.
We assume without loss of generality that the intersection of all reflecting hyperplanes of~$W$ is reduced to~$0$.

A \defn{root system} for~$W$ is a set~$\Phi$ of vectors stable under the action of~$W$ and containing precisely two opposite vectors orthogonal to each reflection hyperplane of~$W$.
Fix a linear functional~$f:V \to \R$ such that $f(\beta) \ne 0$ for all~$\beta \in \Phi$.
It splits the root system~$\Phi$ into the set of \defn{positive roots}~$\Phi^+ \eqdef \set{\beta \in \Phi}{f(\beta)>0}$ and the set of \defn{negative roots}~$\Phi^- \eqdef -\Phi^+$.
The \defn{simple roots} are the roots which lie on the extremal rays of the cone generated by~$\Phi^+$.
They form a basis~$\Delta$ of the vector space~$V$.
The \defn{simple reflections}~$S \eqdef \set{s_\alpha}{\alpha \in \Delta}$ generate the Coxeter group~$W$.
The pair~$(W,S)$ is a \defn{finite Coxeter system}.
For~$s \in S$, we let~$\alpha_s \in \Delta$ be the simple root orthogonal to the reflecting hyperplane of~$s$.

The \defn{length} of an element~$w \in W$ is the length~$\ell(w)$ of the smallest expression of~$w$ as a product of the generators in~$S$.
An expression~$w = s_1 \cdots s_p$, with $s_1, \dots, s_p \in S$, is \defn{reduced} if~$p = \ell(w)$.
The length of~$w$ is also known to be the cardinality of the \defn{inversion set} of~$w$, defined as the set~$\inv(w) \eqdef \Phi^+ \cap w(\Phi^-)$ of positive roots sent to negative roots by~$w^{-1}$.
Indeed, $\inv(w) = \{ \alpha_{s_1}, s_1(\alpha_{s_2}), \dots, s_1 \cdots s_{\ell-1}(\alpha_{s_\ell}) \}$ for any reduced expression~$w = s_1 \cdots s_\ell$ of~$w$.
The \defn{(right) weak order} is the partial order on~$W$ defined by~$u \le w$ if there exists~$v \in W$ with~$uv = w$ and~$\length(u)+\length(v)=\length(w)$.
In other words, $u \le v$ if and only if $\inv(u) \subseteq \inv(v)$.

\begin{example}[Type~$A$ --- Symmetric groups]
The symmetric group~$\fS_{n+1}$, acting on the linear hyperplane $\one^\perp \eqdef \set{x \in \R^{n+1}}{\dotprod{\one}{x} = 0}$ by permutation of the coordinates, is the reflection group of \defn{type~$A_n$}.
It is the group of isometries of the standard $n$-dimensional regular simplex $\conv \{e_1,\dots,e_{n+1}\}$.
Its reflections are the transpositions of~$\fS_{n+1}$ and the set~$\set{e_i-e_j}{i \neq j}$ is a root system for~$A_n$.
We can choose the linear functional~$f$ such that the simple reflections are the adjacent transpositions~$\tau_i \eqdef (i\;\;i+1)$, and the simple roots are the vectors~$e_{i+1}-e_i$.
\end{example}


\subsection{Subword complexes}
\label{subsec:subwordComplex}

We consider a finite Coxeter system~$(W,S)$, a word $\Q \eqdef \q_1\q_2 \cdots \q_m$ on the generators of~$S$, and an element~$\rho \in W$.
A.~Knutson and E.~Miller~\cite{KnutsonMiller-subwordComplex} define the \defn{subword complex}~$\subwordComplex$ to be the simplicial complex of those subwords of~$\Q$ whose complements contain a reduced expression for~$\rho$ as a subword.
A vertex of~$\subwordComplex$ is a position of a letter in~$\Q$.
We denote by~$[m] \eqdef \{1,2,\dots,m\}$ the set of positions in~$\Q$.
A facet of~$\subwordComplex$ is the complement of a set of positions which forms a reduced expression for~$\rho$ in~$\Q$.
We denote by~$\facets$ the set of facets of~$\subwordComplex$.
We write~$\rho \prec \Q$ when~$\Q$ contains a reduced expression of~$\rho$, \ie when~$\subwordComplex$ is non-empty.

\begin{example}
\label{exm:toto}
Consider the type~$A$ Coxeter group~$\fS_4$ generated by~$\{ \tau_1,\tau_2,\tau_3\}$.
Let $\Qex \eqdef \sq{\tau}_2\sq{\tau}_3\sq{\tau}_1\sq{\tau}_3\sq{\tau}_2\sq{\tau}_1\sq{\tau}_2\sq{\tau}_3\sq{\tau}_1$ and~$\rhoex \eqdef [4,1,3,2]$.
The reduced expressions of~$\rhoex$ are $\tau_2\tau_3\tau_2\tau_1$, $\tau_3\tau_2\tau_3\tau_1$, and $\tau_3\tau_2\tau_1\tau_3$.
Thus, the facets of the subword complex $\subwordComplex[\Qex,\rhoex]$ are given by $\{1, 2, 3, 5, 6\}$, $\{1, 2, 3, 6, 7\}$, $\{1, 2, 3, 7, 9\}$, $\{1, 3, 4, 5, 6\}$, $\{1, 3, 4, 6, 7\}$, $\{1, 3, 4, 7, 9\}$, $\{2, 3, 5, 6, 8\}$, $\{2, 3, 6, 7, 8\}$, $\{2, 3, 7, 8, 9\}$, $\{3, 4, 5, 6, 8\}$, $\{3, 4, 6, 7, 8\}$, and $\{3, 4, 7, 8, 9\}$.
Let $\Iex \eqdef \{1,3,4,7,9\}$ and $\Jex \eqdef \{3,4,7,8,9\}$ denote two facets of $\subwordComplex[\Qex,\rhoex]$.
We will use this example throughout this paper to illustrate further notions.
\end{example}

\begin{example}[Type~$A$ --- Primitive networks and pseudoline arrangements]
\label{exm:typeA}
For type~$A$ Coxeter systems, subword complexes can be visually interpreted using primitive networks.
A \defn{network}~$\cN$ is a collection of~$n+1$ horizontal lines (called \defn{levels}, and labeled from bottom to top), together with~$m$ vertical segments (called \defn{commutators}, and labeled from left to right) joining two different levels and such that no two of them have a common endpoint.
We only consider \defn{primitive} networks, where any commutator joins two consecutive levels.
See \fref{fig:network} (left).

\begin{figure}[b]
  \centerline{\includegraphics[width=.9\textwidth]{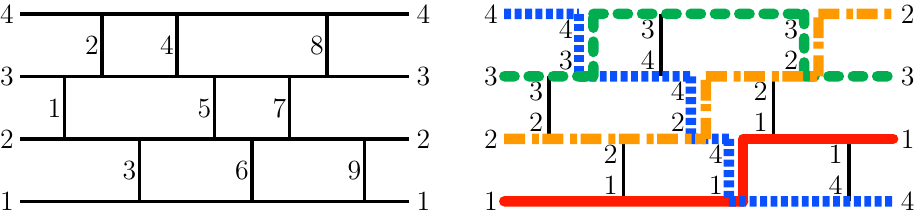}}
  \caption{The network~$\cN_{\Qex}$ (left) and the pseudoline arrangement~$\Lambda_{\Iex}$ for the facet~$\Iex = \{1,3,4,7,9\}$ of~$\subwordComplex[\Qex,\rhoex]$ (right).}
  \label{fig:network}
\end{figure}

A \defn{pseudoline} supported by the network~$\cN$ is an abscissa monotone path on~$\cN$.
A commutator of~$\cN$ is a \defn{crossing} between two pseudolines if it is traversed by both pseudolines, and a \defn{contact} if its endpoints are contained one in each pseudoline.
A \defn{pseudoline arrangement}~$\Lambda$ is a set of~$n+1$ pseudolines on~$\cN$, any two of which have at most one crossing, possibly some contacts, and no other intersection.
We label the pseudolines of~$\Lambda$ from bottom to top on the left of the network, and we define~$\pi(\Lambda) \in \fS_{n+1}$ to be the permutation given by the order of these pseudolines on the right of the network.
Note that the crossings of~$\Lambda$ correspond to the inversions of~$\pi(\Lambda)$.
See \fref{fig:network} (right).

Consider the type~$A$ Coxeter group~$\fS_{n+1}$ generated by~$S = \set{\tau_i}{i \in [n]}$, where~$\tau_i$ is the adjacent transposition~$(i\;\;i+1)$.
To a word~$\Q \eqdef \q_1\q_2 \cdots \q_m$ with~$m$ letters on~$S$, we associate a primitive network~$\cN_\Q$ with~$n+1$ levels and~$m$ commutators.
If~$\q_j = \sq{\tau}_p$, the $j$\ordinal{} commutator of~$\cN_\Q$ is located between the $p$\ordinal{} and $(p+1)$\ordinal{} levels of~$\cN_\Q$.
See \fref{fig:network} (left).
For~$\rho \in \fS_{n+1}$, a facet~$I$  of~$\subwordComplex$ corresponds to a pseudoline arrangement~$\Lambda_I$ supported by~$\cN_\Q$ and with~$\pi(\Lambda_I) = \rho$.
The positions of the contacts (resp.~crossings) of~$\Lambda_I$ correspond to the positions of~$I$ (resp.~of the complement of~$I$).
See \fref{fig:network} (right).
\end{example}

\begin{example}[Combinatorial models for geometric graphs]
\label{exm:geometricGraphs}
As pointed out in \cite{PilaudPocchiola}, pseudoline arrangements on primitive networks give combinatorial models for the following families of geometric graphs (see \fref{fig:geometricGraphs}):
\begin{enumerate}[(i)]
\item triangulations of convex polygons;
\item multitriangulations of convex polygons~\cite{PilaudSantos-multitriangulations};
\item pointed pseudotriangulations of points in general position in the plane~\cite{RoteSantosStreinu-survey};
\item pseudotriangulations of disjoint convex bodies in the plane~\cite{PocchiolaVegter}.
\end{enumerate}

\begin{figure}[b]
  \vspace*{-2pt}
  \centerline{\includegraphics[width=\textwidth]{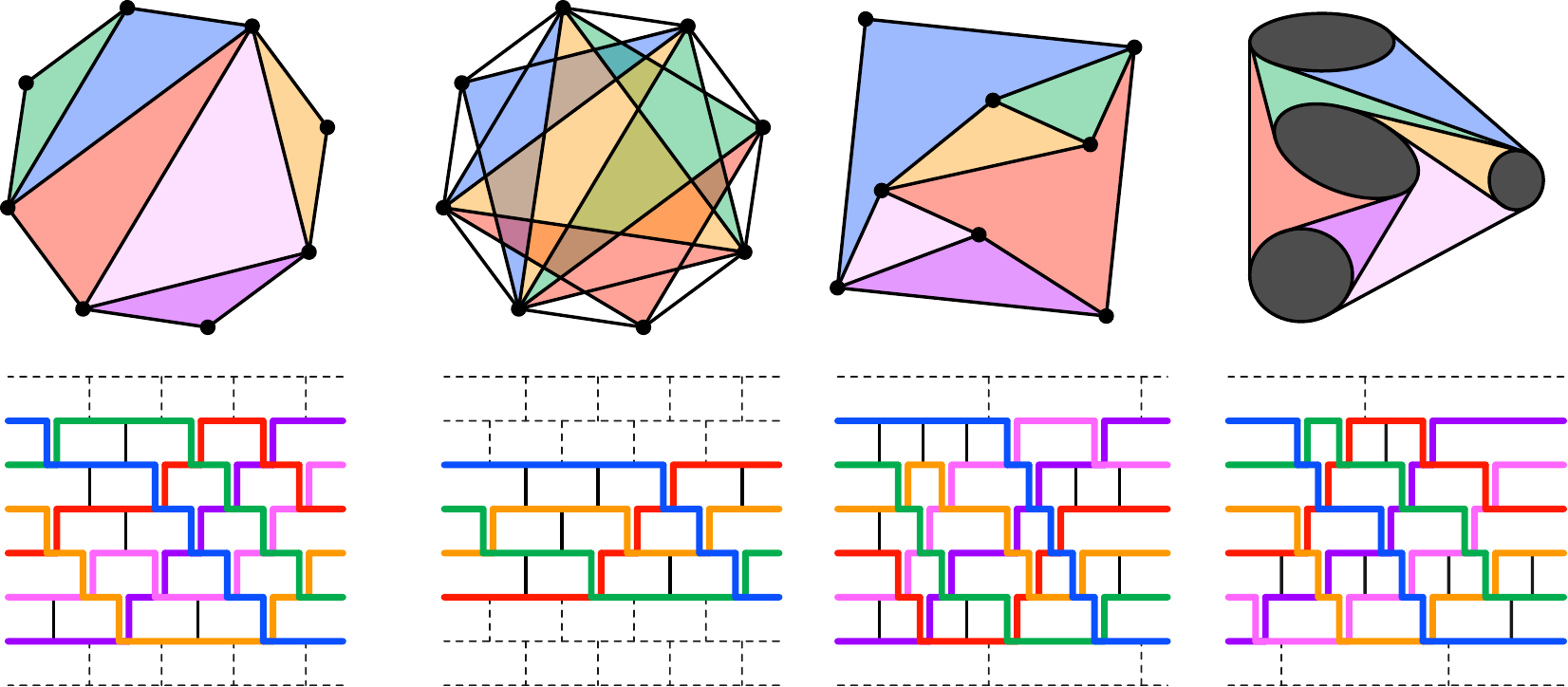}}
  \caption{Primitive sorting networks are combinatorial models for triangulations, multitriangulations, and pseudotriangulations of points or disjoint convex bodies.}
  \label{fig:geometricGraphs}
\end{figure}

For example, consider a triangulation~$T$ of a convex~$(n+3)$-gon.
Define the direction of a line of the plane to be the angle~$\theta \in [0,\pi)$ of this line with the horizontal axis.
Define also a bisector of a triangle~$\triangle$ to be a line passing through a vertex of~$\triangle$ and separating the other two vertices of~$\triangle$.
For any direction~$\theta \in [0,\pi)$, each triangle of~$T$ has precisely one bisector in direction~$\theta$.
We can thus order the~$n+1$ triangles of~$T$ according to the order~$\pi_\theta$ of their bisectors in direction~$\theta$.
The pseudoline arrangement associated to~$T$ is then given by the evolution of the order~$\pi_\theta$ when the direction~$\theta$ describes the interval~$[0,\pi)$.
A similar duality holds for the other three families of graphs, replacing triangles by the natural cells decomposing the geometric graph (stars for multitriangulations~\cite{PilaudSantos-multitriangulations}, or pseudotriangles for pseudotriangulations~\cite{RoteSantosStreinu-survey}).
See \fref{fig:geometricGraphs} for an illustration.
Details can be found in~\cite{PilaudPocchiola}.
\end{example}

\begin{remark}
\label{rem:reversal}
There is a natural reversal operation on subword complexes.
Namely,
$$\subwordComplex[\q_m \cdots \q_1, \rho^{-1}] = \set{\set{m+1-i}{i \in I}}{I \in \subwordComplex[\q_1 \cdots \q_m, \rho]}.$$
We will use this operation to relate positive and negative labelings, facets and trees.
\end{remark}


\subsection{Inductive structure}
\label{subsec:induction}

We denote by~$\Q_\eraseFirst \eqdef \q_2 \cdots \q_m$ and $\Q_\eraseLast \eqdef \q_1 \cdots \q_{m-1}$ the words on~$S$ obtained from $\Q \eqdef \q_1 \cdots \q_m$ by deleting its first and last letters, respectively.
We denote by $\shiftRight{X}$ the right shift $\set{x+1}{x \in X}$ of a subset~$X$ of~$\Z$.
For a collection~$\cX$ of subsets of~$\Z$, we write $\shiftRight{\cX}$ for the set $\set{\shiftRight{X}}{X \in \cX}$.
Moreover, we denote by $\cX \join z$ (or by $z \join \cX$) the join $\set{X \cup z}{X \in \cX}$ of~$\cX$ with some $z \in \Z$.
Remember that $\ell(\rho)$ denotes the length of~$\rho$ and that we write~$\rho \prec \Q$ when~$\Q$ contains a reduced expression of~$\rho$.

We can decompose inductively the facets of the subword complex~$\subwordComplex$ depending on whether or not they contain the last letter of~$\Q$.
Denoting by~$\varepsilon$ the empty word and by~$e$ the identity of~$W$, we have~$\facets[\varepsilon, e] = \{\varnothing\}$ and~$\facets[\varepsilon, \rho] = \varnothing$ if~$\rho \ne e$.
Moreover, for a non-empty word~$\Q$ on~$S$, the set~$\facets$ is given by
\begin{enumerate}[(i)]
\item $\facets[\Q_\eraseLast, \rho q_m]$ if~$m$ appears in none of the facets of~$\subwordComplex$ (\ie if~$\rho \not\prec \Q_\eraseLast$);
\item $\facets[\Q_\eraseLast, \rho] \join m$ if~$m$ appears in all the facets of~$\subwordComplex$ (\ie if~$\length(\rho q_m) > \length(\rho)$);
\item $\facets[\Q_\eraseLast,\rho q_m] \, \sqcup \, \big( \facets[\Q_\eraseLast,\rho] \join m \big)$ otherwise.
\end{enumerate}
By reversal (see Remark~\ref{rem:reversal}), there is also a similar inductive decomposition of the facets of the subword complex~$\subwordComplex$ depending on whether or not they contain the first letter of~$\Q$.
Namely, for a non-empty word~$\Q$, the set~$\facets$ is given by
\begin{enumerate}[(i)]
\item $\shiftRight{\facets[\Q_\eraseFirst, q_1 \rho]}$ if~$1$ appears in none of the facets of~$\subwordComplex$ (\ie if~$\rho \not\prec \Q_\eraseFirst$);
\item $1 \join \shiftRight{\facets[\Q_\eraseFirst, \rho]}$ if~$1$ appears in all the facets of~$\subwordComplex$ (\ie if~$\length(q_1\rho) > \length(\rho)$);
\item $\shiftRight{\facets[\Q_\eraseFirst,q_1\rho]} \, \sqcup \, \big( 1 \join \shiftRight{\facets[\Q_\eraseFirst,\rho]} \big)$ otherwise.
\end{enumerate}
Although we will only use these decompositions for the facets~$\facets$, they extend to the whole subword complex~$\subwordComplex$ and are used to obtain the following result.

\begin{theorem}[\protect{\cite[Corollary~3.8]{KnutsonMiller-subwordComplex}}]
\label{theo:KnutsonMiller}
The subword complex~$\subwordComplex$ is either a simplicial sphere or a simplicial ball.
\end{theorem}


\subsection{Flips and roots}
\label{subsec:flips&roots}

Let~$I$ be a facet of~$\subwordComplex$ and~$i$ be a position in~$I$.
If there exists a facet~$J$ of $\subwordComplex$ and a position~$j \in J$ such that~$I \ssm i = J \ssm j$, we say that $I$~and~$J$ are \defn{adjacent} facets, that~$i$ is \defn{flippable} in~$I$, and that~$J$ is obtained from~$I$ by \defn{flipping}~$i$.
Note that, if they exist,~$J$ and~$j$ are unique by Theorem~\ref{theo:KnutsonMiller}.
We say that the flip from $I$ to $J$ \defn{flips out} $i$ and \defn{flips in} $j$.

We denote by~$\flipGraph$ the graph of flips, whose vertices are the facets of~$\subwordComplex$ and whose edges are pairs of adjacent facets.
That is, $\flipGraph$~is the ridge graph of the simplicial complex~$\subwordComplex$.
This graph is connected according to Theorem~\ref{theo:KnutsonMiller}.

This graph can be naturally oriented by the direction of the flips as follows.
Let~$I$ and~$J$ be two adjacent facets of~$\subwordComplex$ with~$I \ssm i = J \ssm j$.
We say that the flip from~$I$ to~$J$ is \defn{increasing} if~$i<j$.
We consider the flip graph $\flipGraph$ oriented by increasing flips.

\begin{example}
\fref{fig:flipGraph} represents the increasing flip graph~$\flipGraph[\Qex,\rhoex]$ for the subword complex~$\subwordComplex[\Qex, \rhoex]$ of Example~\ref{exm:toto}.
The facets of~$\subwordComplex[\Qex,\rhoex]$ appear in lexicographic order from left to right.
Thus, all flips are increasing from left~to~right.

\begin{figure}
  \centerline{\includegraphics[width=\textwidth]{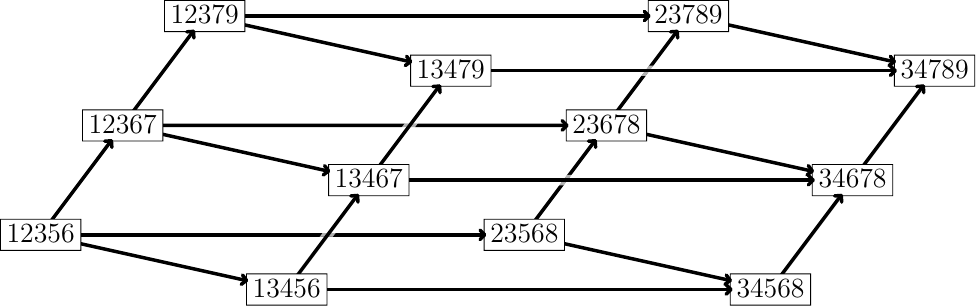}}
  \caption{The increasing flip graph~$\flipGraph[\Qex,\rhoex]$.}
  \label{fig:flipGraph}
\end{figure}
\end{example}

\begin{remark}
The increasing flip graph of~$\subwordComplex$ was already considered by A. Knutson and E.~Miller~\cite[Remark~4.5]{KnutsonMiller-subwordComplex}.
It carries various combinatorial informations about the subword complex~$\subwordComplex$.
In particular, since the lexicographic ordering of the facets of~$\subwordComplex$ is a shelling order for~$\subwordComplex$, the $h$-vector of the subword complex~$\subwordComplex$ is the in-degree sequence of the increasing flip graph~$\flipGraph$.
\end{remark}

Throughout the paper, we consider flips as elementary operations on subword complexes.
In practice, the necessary information to perform flips in a facet~$I$ of~$\subwordComplex$ is encoded in its root function~$\Root{I}{\cdot} : [m] \to \Phi$ defined by
$$\Root{I}{k} \eqdef \wordprod{\Q}{[k-1] \ssm I}(\alpha_{q_k}),$$
where~$\wordprod{\Q}{X}$ denotes the product of the reflections~$q_x \in \Q$ for~$x \in X$.
The \defn{root configuration} of the facet~$I$ is the multiset~$\Roots{I} \eqdef \multiset{\Root{I}{i}}{i \text{ flippable in } I}$.
The root function was introduced by C.~Ceballos, J.-P.~Labb\'e and C.~Stump~\cite{CeballosLabbeStump}, and we extensively studied root configurations in~\cite{PilaudStump} in the construction of brick polytopes for spherical subword complexes.
The main properties of the root function are summarized in the following proposition, whose proof is similar to that of~\cite[Lemmas~3.3 and 3.6]{CeballosLabbeStump} or~\cite[Lemma~3.3]{PilaudStump}.

\begin{proposition}
\label{prop:roots&flips}
Let~$I$ be any facet of the subword complex~$\subwordComplex$.
\begin{enumerate}
\item
\label{prop:roots&flips:inversions}
The map~$\Root{I}{\cdot}:i \mapsto \Root{I}{i}$ is a bijection from the complement of~$I$ to the inversion set of~$\rho$.

\item
\label{prop:roots&flips:flippable}
The map~$\Root{I}{\cdot}$ sends the flippable positions in~$I$ to~$\set{\pm \beta}{\beta \in \inv(\rho)}$ and the unflippable ones to~$\Phi^+ \ssm \inv(\rho)$.

\item
\label{prop:roots&flips:flip}
If $I$~and~$J$ are two adjacent facets of~$\subwordComplex$ with~$I \ssm i = J \ssm j$, the position $j$ is the unique position in the complement of~$I$ for which~${\Root{I}{j} \in \{\pm\Root{I}{i}}\}$.

\item
\label{prop:roots&flips:increasing}
In the situation of\;\eqref{prop:roots&flips:flip}, we have~${\Root{I}{i} = \Root{I}{j} \in \Phi^+}$ if~$i < j$ (increasing flip), while ${\Root{I}{i} = -\Root{I}{j} \in \Phi^-}$ if~$i > j$ (decreasing flip).

\item
\label{prop:roots&flips:update}
In the situation of\;\eqref{prop:roots&flips:flip}, the map~$\Root{J}{\cdot}$ is obtained from the map~$\Root{I}{\cdot}$ by:
$$\Root{J}{k} = \begin{cases} s_{\Root{I}{i}}(\Root{I}{k}) & \text{if } \min(i,j) < k \le \max(i,j), \\ \Root{I}{k} & \text{otherwise.} \end{cases}$$
\end{enumerate}
\end{proposition}
We call $\Root{I}{i} = -\Root{J}{j}$ the \defn{direction} of the flip from the facet $I$ to the facet $J$.

\begin{example}
In type~$A$, roots and flips can easily be described using the primitive network interpretation presented in Example~\ref{exm:typeA}.
Consider a word~$\Q$ on the simple reflections~$\set{\tau_i}{i \in [n]}$, an element~$\rho \in \fS_{n+1}$, and a facet~$I$ of~$\subwordComplex$.
For any~$k \in [m]$, the root~$\Root{I}{k}$ is the difference~$e_t-e_b$ where~$t$ and~$b$ are the indices of the pseudolines of~$\Lambda_I$ which arrive respectively on the top and bottom endpoints of the $k$\ordinal{} commutator of~$\cN_\Q$.
A flip exchanges a contact between two pseudolines~$t$ and~$b$ of~$\Lambda_I$ with the unique crossing between~$t$ and~$b$ in~$\Lambda_I$ (when it exists).
Such a flip is increasing if the contact lies before the crossing, \ie if~$t > b$.
\fref{fig:flip} illustrates the properties of Proposition~\ref{prop:roots&flips} on the subword complex~$\subwordComplex[\Qex, \rhoex]$ of Example~\ref{exm:toto}.

\begin{figure}[ht]
  \centerline{\includegraphics[width=.9\textwidth]{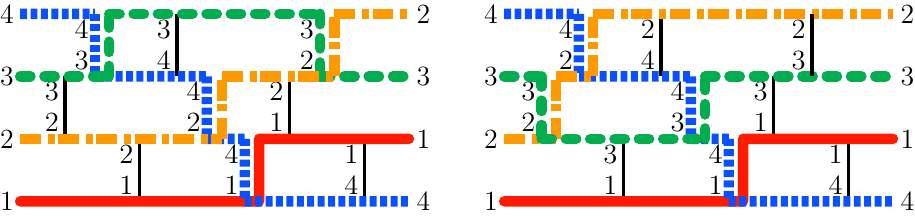}}
  \caption{The increasing flip from facet~$\Iex = \{1,3,4,7,9\}$ to facet~$\Jex = \{3,4,7,8,9\}$ of the subword complex~$\subwordComplex[\Qex,\rhoex]$, illustrated on the network~$\cN_{\Qex}$.}
  \label{fig:flip}
\end{figure}
\end{example}


\subsection{Restriction of subword complexes to parabolic subgroups}
\label{subsec:restriction}

In the proof of our main result, we will need to restrict subword complexes to dihedral parabolic subsystems.
The following statement can essentially be found in~\cite[Proposition~3.7]{PilaudStump}, we provide a proof here as well for the sake of completeness.

\begin{proposition}[\protect{\cite[Proposition~3.7]{PilaudStump}}]
\label{prop:restriction}
Let~$\subwordComplex$ be a subword complex for a Coxeter system $(W,S)$ acting on~$V$, and let $V' \subseteq V$ be a subspace of $V$.
The simplicial complex given by all facets $J$ of $\subwordComplex$ reachable from a particular facet~$I$ by flips whose directions are contained in $V'$ is isomorphic to a subword complex $\subwordComplex[\Q',\rho']$ for the restriction of $(W,S)$ to $V'$.
The order of the letters is preserved by this isomorphism.
In particular, the restriction of the increasing flip graph $\flipGraph$ to these reachable facets is isomorphic to the increasing flip graph~$\flipGraph[\Q',\rho']$.
\end{proposition}

\begin{proof}
To prove this proposition, we explicitly construct the word~$\Q'$ on~$S'$ and the element~$\rho' \in W'$ where $(W',S')$ is the restriction of $(W,S)$ to the subspace~$V'$.

First, the element~$\rho'$ only depends on $\rho$ and on~$V'$: it is given by the projection of $\rho$ onto $W'$.
This is to say that $\rho'$ is the unique element in~$W'$ whose inversion set is~$\inv(\rho') = \inv(\rho) \cap V'$.
To see that $\inv(\rho) \cap V'$ is again an inversion set, remember that a subset $\mathcal{I}$ of $\Phi^+$ is an inversion set for an element in $W$ if and only if for all $\alpha,\beta,\gamma \in \Phi^+$ such that~$\gamma = a\alpha + b\beta$ for some $a,b\in \R_{\geq 0}$,
$$\alpha,\beta \in \mathcal{I} \; \implies \; \gamma \in \mathcal{I} \; \implies \; \big( \alpha \in \mathcal{I} \text{ or } \beta \in \mathcal{I}\big),$$
see \eg \cite{Papi}. Moreover this property is preserved under intersection with linear subspaces.

We now construct the word~$\Q'$ and the facet~$I'$ of $\subwordComplex[\Q',\rho']$ corresponding to the particular facet $I$ of $\subwordComplex$.
For this, let $X \eqdef \{ x_1,\ldots,x_p\}$ be the set of positions~$k \in [m]$ such that $\Root{I}{k} \in V'$.
The word~$\Q'$ has~$p$ letters corresponding to the positions in~$X$, and the facet~$I'$ contains precisely the positions~$k \in [p]$ such that the position~$x_k$ is in~$I$.
To construct the word~$\Q'$, we scan~$\Q$ from left to right as follows.
We initialize~$\Q'$ to the empty word, and for each $1 \leq k \leq p$, we add a letter $\q'_k \in S'$ to~$\Q'$ in such a way that $\Root{I'}{k} = \Root{I}{x_k}$.
To see that such a letter exists, we distinguish two cases.
Assume first that~$\Root{I}{x_k}$ is a positive root.
Let~$\mathcal{I}$ be the inversion set of~$w \eqdef \wordprod{\Q}{[x_k-1] \ssm I}$ and~$\mathcal{I}' = \mathcal{I} \cap V'$ be the inversion set of~$w' \eqdef \wordprod{\Q'}{[k-1] \ssm I'}$.
Then the set~$\mathcal{I}' \cup \{\Root{I}{x_k}\}$ is again an inversion set (as the intersection of~$V'$ with the inversion set~$\mathcal{I} \cup \{\Root{I}{x_k}\}$ of~$w q_{x_k}$) which contains the inversion set~$\mathcal{I'}$ of~$w'$ together with a unique additional root.
Therefore, the corresponding element of~$W'$ can be written as~$w'q'_k$ for some simple reflection~$q'_k \in S'$.
Assume now that~$\Root{I}{x_k}$ is a negative root.
Then~$x_k \in I$, so that we can flip it with a position~$x_{k'} < x_k$, and we can then argue on the resulting facet.

By the procedure described above, we eventually obtain the subword complex $\subwordComplex[\Q',\rho']$ and its facet $I'$ corresponding to the facet $I$. Finally observe that sequences of flips in $\subwordComplex$ starting at the facet $I$, and whose directions are contained in $V'$, correspond bijectively to sequences of flips in $\subwordComplex[\Q',\rho']$ starting at the facet $I'$. In particular, let $J$ and $J'$ be two facets reached from $I$ and from $I'$, respectively, by such a sequence. We then have that the root configuration of $J'$ is exactly the root configuration of $J$ intersected with $V'$, and that the order in which the roots appear in the root configurations is preserved. This completes the proof.
\end{proof}

\begin{example}
\begin{figure}[t]
  \centerline{\includegraphics[width=.95\textwidth]{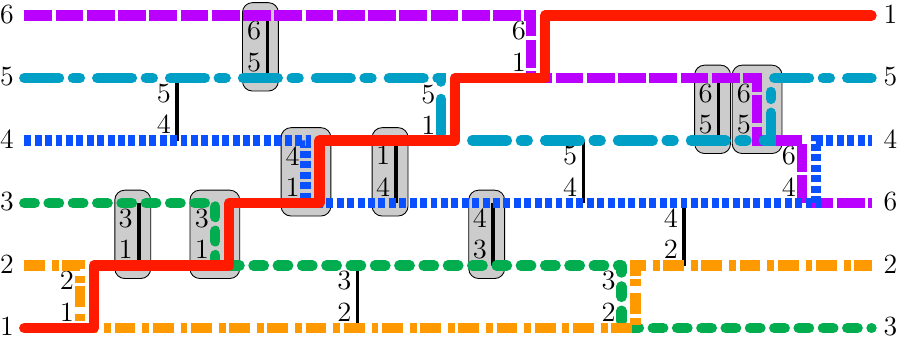}}
  \vspace{.1cm}
  \centerline{$\Downarrow$ \quad restriction to the space~$V' = \vect \langle e_3-e_1, e_4-e_3, e_6-e_5 \rangle$ \quad $\Downarrow$}
  \vspace{.1cm}
  \centerline{\includegraphics[width=.95\textwidth]{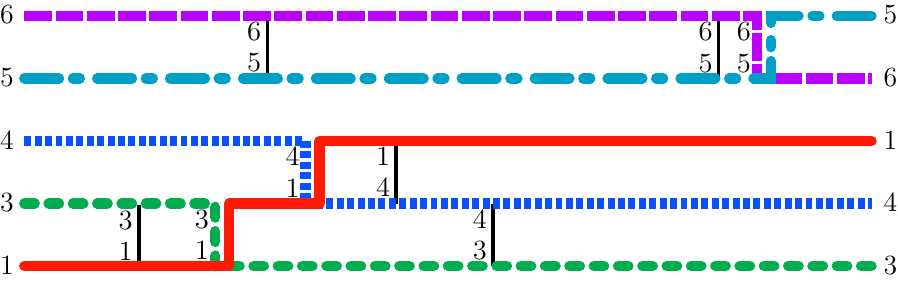}}
  \caption{Restricting subword complexes.}
  \label{fig:restriction}
\end{figure}

To illustrate different possible situations happening in this restriction, we consider the subword complex~$\subwordComplex$ on the Coxeter group~$A_5 = \fS_6$ generated by~$S = \{\tau_1, \dots, \tau_5\}$, the word $\Q \eqdef \tau_1 \tau_2 \tau_4 \tau_2 \tau_5 \tau_3 \tau_1 \tau_3 \tau_4 \tau_2 \tau_5 \tau_3 \tau_1 \tau_2 \tau_4 \tau_4 \tau_3$ and the element~$\rho \eqdef [3,2,6,4,5,1] = \tau_1 \tau_2 \tau_3 \tau_4 \tau_5 \tau_1 \tau_4 \tau_3$. The sorting network corresponding to the subword complex~$\subwordComplex$ and the pseudoline arrangement corresponding to the facet~$I \eqdef \{2,3,5,7,8,10,12,14,15\}$ of~$\subwordComplex$ are shown in \fref{fig:restriction}\,(top).
Let~$V'$ be the subspace of~$V$ spanned by the roots $e_3-e_1$, $e_4-e_3$ and $e_6-e_5$.
Let~$X = \{x_1,\dots,x_8\} = \{2,4,5,6,8,10,15,16\}$ denote the set of positions~$k \in [17]$ for which~$\Root{I}{k} \in V'$.
These positions are circled in \fref{fig:restriction}\,(top).

We can now directly read off the subword complex~$\subwordComplex[\Q',\rho']$ corresponding to the restriction of~$\subwordComplex$ to all facets reachable from~$I$ by flips with directions in~$V'$.
Namely, the restriction of~$(W,S)$ to~$V'$ is the Coxeter system~$(W',S')$ where~$W'$ is generated by $S' = \{\tau'_1,\tau'_2,\tau'_3\} = \{ (1\;3), (3\;4), (5\;6)\}$, and thus of type $A_2 \times A_1$.
Moreover, we have $\Q' = \tau'_1 \tau'_1 \tau'_3 \tau'_2 \tau'_2 \tau'_1 \tau'_3 \tau'_3$, corresponding to the roots at positions in~$X$, and $\rho' = \tau'_1\tau'_2\tau'_3$, with inversion set given by the positive roots corresponding to the roots at positions in~$X \ssm I$.
Finally, the facet~$I'$ corresponding to~$I$ is given by~$I' = \{1,3,5,6,7\}$.
The sorting network corresponding to the restricted subword complex~$\subwordComplex[\Q',\rho']$ and the pseudoline arrangement corresponding to the facet~$I'$ of~$\subwordComplex[\Q',\rho']$ are shown in \fref{fig:restriction}\,(bottom).

As stated in Proposition~\ref{prop:restriction}, the map which sends a facet~$J$ of~$\subwordComplex[\Q',\rho']$ to the facet~$\set{x_j}{j \in J} \cup (I \ssm X)$ of~$\subwordComplex$ defines an isomorphism between the increasing flip graph $\flipGraph[\Q',\rho']$ and the restriction of the increasing flip graph $\flipGraph$ to all facets reachable from $I$ by flips with directions~in~$V'$.
\end{example}


\section{EL-labelings and spanning trees for the subword complex}
\label{sec:labelings&trees}

\subsection{EL-labelings of the increasing flip graph}
\label{subsec:ELlabelingsflipgraph}

We now define two natural edge labelings of the increasing flip graph~$\flipGraph$.

Let~$I$ and~$J$ be two adjacent facets of~$\subwordComplex$, with $I \ssm i = J \ssm j$ and~${i<j}$.
We label the edge $I \edge J$ of~$\flipGraph$ with the positive edge label~$\positiveEdgeLabel(I \sep J) \eqdef i$ and with the negative edge label~$\negativeEdgeLabel(I \sep J) \eqdef j$.
In other words, $\positiveEdgeLabel$ labels the position flipped out while $\negativeEdgeLabel$ labels the position flipped in during the flip~$I \edge J$.
We call $\positiveEdgeLabel : E(\flipGraph) \to [m]$ the \defn{positive edge labeling} and $\negativeEdgeLabel : E(\flipGraph) \to [m]$ the \defn{negative edge labeling} of the increasing flip graph~$\flipGraph$.
The terms ``positive'' and ``negative'' emphasize the fact that the roots~$\Root{I}{\positiveEdgeLabel(I \sep J)}$ and~$\Root{J}{\negativeEdgeLabel(I \sep J)}$ are always positive and negative roots respectively.

The positive and negative edge labelings are reverse to one another (see Remark~\ref{rem:reversal}).
Namely, $I \edge J$ is an edge in the increasing flip graph $\flipGraph[\q_m \cdots \q_1, \rho^{-1}]$ if and only if $J' \eqdef \set{m+1-j}{j \in J} \edge I' \eqdef \set{m+1-i}{i \in I}$ is an edge in the increasing flip graph $\flipGraph[\q_1 \cdots \q_m, \rho]$, and in this case $\negativeEdgeLabel(I \sep J) = m+1-\positiveEdgeLabel(J' \sep I')$.
However, we will work in parallel with both labelings, since we believe that certain results are simpler to present on the positive side while others are simpler on the negative side.
We always provide proofs on the easier side and leave it to the reader to translate to the opposite side.

\begin{example}
Consider the subword complex~$\subwordComplex[\Qex,\rhoex]$ of Example~\ref{exm:toto}.
We have represented on \fref{fig:ELlabeling} the positive and negative edge labelings~$\positiveEdgeLabel$ and~$\negativeEdgeLabel$.
Since we have represented the graph~$\flipGraph[\Qex,\rhoex]$ such that the flips are increasing from left to right, each edge has its positive label on the left and its negative label on the right.

\begin{figure}[ht]
  \centerline{\includegraphics[width=\textwidth]{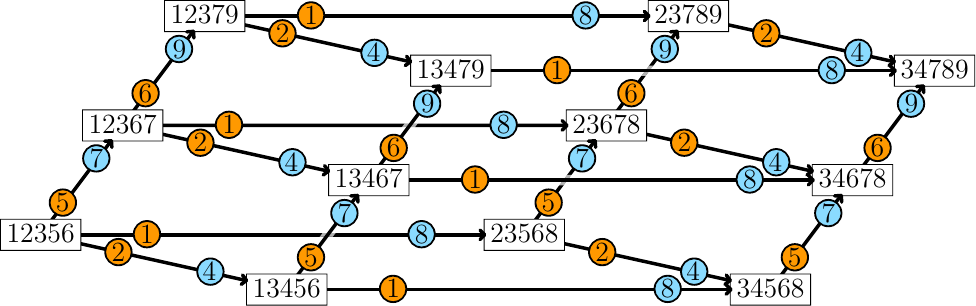}}
  \caption{The positive and negative edge labelings~$\positiveEdgeLabel$ and~$\negativeEdgeLabel$ of~$\flipGraph[\Qex,\rhoex]$. Each edge has its positive label on the left (orange) and its negative label on the right (blue).}
  \label{fig:ELlabeling}
\end{figure}
\end{example}

The central result of this paper concerns the positive and negative edge labelings of the increasing flip graph.

\begin{theorem}
\label{thm:ELlabeling}
The positive edge labeling~$\positiveEdgeLabel$ and the negative edge labeling~$\negativeEdgeLabel$ are both EL-labelings of the increasing flip graph.
\end{theorem}

For Cambrian lattices, whose Hasse diagrams were shown to be particular cases of increasing flip graphs in~\cite[Section~6]{PilaudStump}, a similar result was recently obtained by M.~Kallipoliti and H.~M\"uhle in~\cite{KallipolitiMuhle}.
See Section~\ref{subsubsec:cambrian} for details.

In Sections~\ref{subsec:greedyFacets} to~\ref{subsec:greedyFlipAlgorithm}, we present applications of Theorem~\ref{thm:ELlabeling} to the construction of canonical spanning trees and to the generation of the facets of the subword complex.
Further combinatorial applications of this theorem are also discussed in Section~\ref{sec:furtherCombinatorialProperties}.
We prove Theorem~\ref{thm:ELlabeling} only for the positive edge labeling~$\positiveEdgeLabel$, and leave it to the reader to translate the proof to the negative edge labeling~$\negativeEdgeLabel$ (through the reversal operation of Remark~\ref{rem:reversal}).
Let~$I$ and~$J$ be two facets of~$\subwordComplex$ such that~$I \directedPath J$.
To show that~$\positiveEdgeLabel$ is indeed an EL-labeling, we have to show that (i) there is a $\positiveEdgeLabel$-rising path from~$I$ to~$J$ in~$\flipGraph$ which is (ii) unique and (iii) lexicographically first among all paths from~$I$ to~$J$ in~$\flipGraph$.
We start with (ii)~and~(iii), which are direct consequences of the following proposition.

\begin{proposition}
\label{prop:unicity}
Let~${I_1 \edge \cdots \edge I_{\ell+1}}$ be a path of increasing flips, and define the labels ${\positiveEdgeLabel_k \eqdef \positiveEdgeLabel(I_k \sep I_{k+1})}$ and $\negativeEdgeLabel_k \eqdef \negativeEdgeLabel(I_k \sep I_{k+1})$.
Then, for all~$k \in [\ell]$, we have
$$\min\{\positiveEdgeLabel_k, \dots, \positiveEdgeLabel_\ell\} = \min(I_k \ssm I_{\ell+1}) \quad \text{and} \quad \max\{\negativeEdgeLabel_1, \dots, \negativeEdgeLabel_k\} = \max(I_{k+1} \ssm I_1).$$
Moreover, the path is $\positiveEdgeLabel$-rising if and only if~$\positiveEdgeLabel_k = \min(I_k \ssm I_{\ell+1})$ for all~$k \in [\ell]$, while the path is $\negativeEdgeLabel$-rising if and only if~$\negativeEdgeLabel_k = \max(I_{k+1} \ssm I_1)$ for all~$k \in [\ell]$.
\end{proposition}

\begin{proof}
The position~$\min\{\positiveEdgeLabel_k, \dots, \positiveEdgeLabel_\ell\}$ is in~$I_k \ssm I_{\ell+1}$ since it is flipped out and never flipped in along the path from~$I_k$ to~$I_{\ell+1}$ (because all flips are increasing).
Moreover, $\min\{\positiveEdgeLabel_k, \dots, \positiveEdgeLabel_\ell\}$~has to coincide with~$\min(I_k \ssm I_{\ell+1})$ otherwise this position would never be flipped out along the path.

This property immediately yields the characterization of $\positiveEdgeLabel$-rising paths.
Indeed, if the path is $\positiveEdgeLabel$-rising, then we have $\positiveEdgeLabel_k = \min(\positiveEdgeLabel_k, \dots, \positiveEdgeLabel_\ell) = \min(I_k \ssm I_{\ell+1})$ for all~$k \in [\ell]$.
Reciprocally, if~$\positiveEdgeLabel_k = \min(I_k \ssm I_{\ell+1})$ for all~$k \in [\ell]$, then we have ${\positiveEdgeLabel_k = \min(I_k \ssm I_{\ell+1}) < \min(I_{k+1} \ssm I_{\ell+1}) = \positiveEdgeLabel_{k+1}}$ so that the path is $\positiveEdgeLabel$-rising.

The proof is similar for the negative edge labeling~$\negativeEdgeLabel$.
\end{proof}

We now need to prove the existence of a $\positiveEdgeLabel$-rising path from~$I$ to~$J$.
Before proving it in full generality, we prove its crucial part in the particular case of dihedral subword complexes.

\begin{lemma}
\label{lem:dihedral}
Let $\subwordComplex$ be a subword complex for a dihedral reflection group $W = I_2(m)$.
Let $I$ and $K$ be two of its facets such that there is a path $I \edge J \edge K$ from $I$ to $K$ in $\flipGraph$ with $\positiveEdgeLabel(I \sep J) > \positiveEdgeLabel(J \sep K)$.
Then there is as well a $\positiveEdgeLabel$-rising path from $I$ to $K$ in $\flipGraph$.
\end{lemma}

\begin{proof}
First, we remark that we construct a path only using letters in~$\Q$ at positions not used in~$I$ (those positions corresponding to the reduced expression for~$\rho$), together with the two positions~$i \eqdef \positiveEdgeLabel(I \sep J)$ and~$j \eqdef \positiveEdgeLabel(J \sep K)$.
Observe here that both~$i$ and~$j$ are already contained in~$I$.

We distinguish two cases: the roots~$\Root{I}{i}$ and~$\Root{I}{j}$ generate either a $1$- or a $2$-dimensional space.
In the first case, we have~$\Root{I}{i} = \Root{I}{j}$ and we can directly flip position $j$ in the facet~$I$ to obtain the facet~$K$.
In the second case, it is straightforward to check that we can perform a $\positiveEdgeLabel$-rising path from~$I$ to~$K$, starting with position~$j$, followed by position~$i$, and finishing by a possibly empty $\positiveEdgeLabel$-rising sequence of flips.
\end{proof}

We are now ready to prove Theorem~\ref{thm:ELlabeling}.
Restricting subword complexes to dihedral parabolic subgroups as presented in Section~\ref{subsec:restriction}, we will reduce the general case to several applications of the dihedral situation treated in Lemma~\ref{lem:dihedral}.

\begin{proof}[Proof of Theorem~\ref{thm:ELlabeling}]
Let~$I$ and~$J$ be two facets of~$\subwordComplex$ related by a path ${I = I_1 \edge \cdots \edge I_{\ell+1} = J}$ of increasing flips.
Let~${\positiveEdgeLabel_k \eqdef \positiveEdgeLabel(I_k \sep I_{k+1})}$.
Assume that this path is not $\positiveEdgeLabel$-rising, and let~$k$ be the smallest index such that ${\positiveEdgeLabel_k \ne \min \{ \positiveEdgeLabel_k, \dots, \positiveEdgeLabel_\ell \}}$, and let~$k' > k$ such that $\positiveEdgeLabel_{k'} = \min \{ \positiveEdgeLabel_k,\dots,\positiveEdgeLabel_\ell \}$.
We now prove that we can flip~$\positiveEdgeLabel_{k'}$ instead of~$\positiveEdgeLabel_{k'-1}$ in~$I_{k'-1}$, and still obtain a path from~$I$ to~$J$ where~$\positiveEdgeLabel_{k'}$ is still smaller than all positive edge labels appearing after it.
In Example~\ref{ex:increasingPath}, we illustrate this procedure on an explicit example.

Clearly~${\positiveEdgeLabel_{k'-1} > \positiveEdgeLabel_{k'}}$, and we have a $\positiveEdgeLabel$-falling sequence of two flips given by ${I_{k'-1} \edge I_{k'} \edge I_{k'+1}}$.
Using Proposition~\ref{prop:restriction}, we can now see these two flips as well in a subword complex for the dihedral parabolic subsystem. For this, restrict $(W,S)$ to the subspace $V'$ spanned by the two roots $\Root{I_{k'-1}}{\positiveEdgeLabel_{k'-1}}$ and $\Root{I_{k'-1}}{\positiveEdgeLabel_{k'}} = \Root{I_{k'}}{\positiveEdgeLabel_{k'}}$. This restricted subword complex corresponds to all facets of $\subwordComplex$ reachable from the particular facet $I_{k'-1}$ by flips whose directions are contained in $V'$.
Applying Lemma~\ref{lem:dihedral}, we can thus replace the subpath ${I_{k'-1} \edge I_{k'} \edge I_{k'+1}}$ by a $\positiveEdgeLabel$-rising path from~$I_{k'-1}$ to~$I_{k'+1}$ flipping first position~$\positiveEdgeLabel_{k'}$ and then a (possibly empty) sequence of positions larger than or equal to~$\positiveEdgeLabel_{k'-1}$.

Repeating this operation, we construct a path from~$I$ to~$J$ such that~$\positiveEdgeLabel_k = \min \{ \positiveEdgeLabel_k, \dots, \positiveEdgeLabel_\ell \}$.
By this procedure, we obtain eventually a $\positiveEdgeLabel$-rising path from~$I$ to~$J$.
This path is unique and lexicographically first among all paths from~$I$ to~$J$ in~$\flipGraph$ according to the characterization given in Proposition~\ref{prop:unicity}.
This concludes the proof that~$\positiveEdgeLabel$ is an EL-labeling of~$\flipGraph$.
The proof is similar for the negative edge labeling~$\negativeEdgeLabel$ (by the reversal operation in~Remark~\ref{rem:reversal}).
\end{proof}

\begin{example}
\label{ex:increasingPath}
Consider the subword complex~$\subwordComplex[\Qex,\rhoex]$ of Example~\ref{exm:toto}, whose labeled increasing flip graph is shown in \fref{fig:ELlabeling}.
Consider the path
$$12356 \edgePositiveLabel{5} 12367 \edgePositiveLabel{6} 12379 \edgePositiveLabel{2} 13479 \edgePositiveLabel{1} 34789,$$
in $\flipGraph$, where the numbers on the arrows are the positive edge labels.
In the language of the proof of Theorem~\ref{thm:ELlabeling}, we have~$k = 1$, $k' = 4$, and therefore we replace the subpath~$12379 \edgePositiveLabel{2} 13479 \edgePositiveLabel{1} 34789$ by the subpath $12379 \edgePositiveLabel{1} 23789 \edgePositiveLabel{2} 34789$, thus obtaining the path
$$12356 \edgePositiveLabel{5} 12367 \edgePositiveLabel{6} 12379 \edgePositiveLabel{1} 23789 \edgePositiveLabel{2} 34789.$$
Applying this operation again and again produces the sequence of paths given by
\begin{gather*}
12356 \edgePositiveLabel{5} 12367 \edgePositiveLabel{1} 23678 \edgePositiveLabel{6} 23789 \edgePositiveLabel{2} 34789, \\
12356 \edgePositiveLabel{1} 23568 \edgePositiveLabel{5} 23678 \edgePositiveLabel{6} 23789 \edgePositiveLabel{2} 34789, \\
12356 \edgePositiveLabel{1} 23568 \edgePositiveLabel{5} 23678 \edgePositiveLabel{2} 34678 \edgePositiveLabel{6} 34789, \\
12356 \edgePositiveLabel{1} 23568 \edgePositiveLabel{2} 34568 \edgePositiveLabel{5} 34678 \edgePositiveLabel{6} 34789.
\end{gather*}
The resulting path is~$\positiveEdgeLabel$-rising.
In this example, all paths happen to have the same length. This does not hold in general, compare \fref{fig:greedyTreeCoxeter} on page~\pageref{fig:greedyTreeCoxeter}, where the path
$123 \edgePositiveLabel{2} 137 \edgePositiveLabel{3} 178 \edgePositiveLabel{1} 678 \edgePositiveLabel{7} 689$
is, for example, replaced by the path
$123 \edgePositiveLabel{2} 137 \edgePositiveLabel{1} 357 \edgePositiveLabel{3} 567 \edgePositiveLabel{5} 678 \edgePositiveLabel{7} 689.$
\end{example}

In contrast to the rising paths, we can have none, one, or more than one $\positiveEdgeLabel$-falling and $\negativeEdgeLabel$-falling paths between two facets~$I$ and~$J$ of~$\subwordComplex$.
Even if we will not need it in the rest of the paper, we observe in the next proposition that there are always as many $\positiveEdgeLabel$-falling paths as $\negativeEdgeLabel$-falling paths from~$I$ to~$J$.
Remember that we say that a path $I_1 \edge I_2 \edge \cdots \edge I_{\ell+1}$ \defn{flips out}  the multiset $\flipsOut \eqdef \multiset{\positiveEdgeLabel(I_k \sep I_{k+1})}{k \in [\ell]}$ and \defn{flips in} the multiset~$\flipsIn \eqdef \multiset{\negativeEdgeLabel(I_k \sep I_{k+1})}{k \in [\ell]}$.
Note that a $\positiveEdgeLabel$-falling (resp.~$\negativeEdgeLabel$-falling) path is determined by the multiset~$\flipsOut$ (resp.~$\flipsIn$) of positions that it flips out (resp.~in).

\begin{proposition}
\label{prop:bijectionFallingPaths}
Let~$I$ and~$J$ be two facets of~$\subwordComplex$.
Then there are as many $\positiveEdgeLabel$-falling paths as $\negativeEdgeLabel$-falling paths from~$I$ to~$J$.
More precisely, for any multisubsets~$\flipsOut$ and~$\flipsIn$ of~$[m]$, there exists a $\positiveEdgeLabel$-falling path from~$I$ to~$J$ which flips out~$\flipsOut$ and flips in~$\flipsIn$, if and only if there exists an $\negativeEdgeLabel$-falling path with the same property.
\end{proposition}

\begin{proof}
Consider a $\positiveEdgeLabel$-falling path~$I = I_1 \edge \cdots \edge I_{\ell+1} = J$.
Define~${\positiveEdgeLabel_k \eqdef \positiveEdgeLabel(I_k \sep I_{k+1})}$ and~${\negativeEdgeLabel_k \eqdef \negativeEdgeLabel(I_k \sep I_{k+1})}$.
We want to prove that there is as well an $\negativeEdgeLabel$-falling path which flips out~$\flipsOut \eqdef \multiset{\positiveEdgeLabel_k}{k \in [\ell]}$ and flips in~$\flipsIn \eqdef \multiset{\negativeEdgeLabel_k}{k \in [\ell]}$.

If the path~$I = I_1 \edge \cdots \edge I_{\ell+1} = J$ happens to be~$\negativeEdgeLabel$-falling, we are done.
Otherwise, consider the first position~$k$ such that~$\negativeEdgeLabel_{k-1} < \negativeEdgeLabel_k$.
Since the path is~$\positiveEdgeLabel$-falling, we thus have~$\positiveEdgeLabel_k < \positiveEdgeLabel_{k-1} < \negativeEdgeLabel_{k-1} < \negativeEdgeLabel_k$.
By Proposition~\ref{prop:roots&flips}\eqref{prop:roots&flips:flip}, we know that~${\Root{I_{k-1}}{\positiveEdgeLabel_{k-1}} = \Root{I_{k-1}}{\negativeEdgeLabel_{k-1}}}$ and~${\Root{I_k}{\positiveEdgeLabel_k} = \Root{I_k}{\negativeEdgeLabel_k}}$.
According to Proposition~\ref{prop:roots&flips}\eqref{prop:roots&flips:update} and to the previous inequalities, we therefore obtain
$$\Root{I_{k-1}}{\positiveEdgeLabel_k} = \Root{I_k}{\positiveEdgeLabel_k} = \Root{I_k}{\negativeEdgeLabel_k} = \Root{I_{k-1}}{\negativeEdgeLabel_k}.$$
Thus, in the facet~$I_{k-1}$, flipping out~$\positiveEdgeLabel_k$ flips in~$\negativeEdgeLabel_k$.
We denote by~$I'_k$ the facet of~$\subwordComplex$ obtained by this flip.
Using again Proposition~\ref{prop:roots&flips}\eqref{prop:roots&flips:update} and the previous inequalities, we obtain that
$$\Root{I'_k}{\positiveEdgeLabel_{k-1}} = s_{\Root{I_{k-1}}{\positiveEdgeLabel_k}}(\Root{I_{k-1}}{\positiveEdgeLabel_{k-1}}) = s_{\Root{I_{k-1}}{\positiveEdgeLabel_k}}(\Root{I_{k-1}}{\negativeEdgeLabel_{k-1}}) = \Root{I'_k}{\negativeEdgeLabel_{k-1}}.$$
Therefore, in the facet~$I'_k$, flipping out~$\positiveEdgeLabel_{k-1}$ flips in~$\negativeEdgeLabel_{k-1}$.
After these two flips, we thus obtain~$I_{k+1}$ (since we flipped out~$\positiveEdgeLabel_k$ and~$\positiveEdgeLabel_{k-1}$, while we flipped in~$\negativeEdgeLabel_k$ and~$\negativeEdgeLabel_{k-1}$).
In other words, we can replace the subpath~${I_{k-1} \edge I_k \edge I_{k+1}}$ by the path~${I_{k-1} \edge I'_k \edge I_{k+1}}$ where we flip first~$\positiveEdgeLabel_k$ to~$\negativeEdgeLabel_k$ and then~$\positiveEdgeLabel_{k-1}$ to~$\negativeEdgeLabel_{k-1}$.
The new path still flips out~$\flipsOut$ and flips in~$\flipsIn$, and the first~$k$ positions it flips in are in decreasing order.
Repeating this transformation finally yields an $\negativeEdgeLabel$-falling path from~$I$ to~$J$ which still flips out~$\flipsOut$ and flips in~$\flipsIn$.
Observe that this path does not necessarily coincide with the~$\positiveEdgeLabel$-falling path we started from.

Since a $\positiveEdgeLabel$-falling (resp.~$\negativeEdgeLabel$-falling) path is determined by the set of positions it flips out (resp.~in), we obtain a bijection between $\positiveEdgeLabel$-falling paths and $\negativeEdgeLabel$-falling paths from~$I$ to~$J$.
They are thus equinumerous.
\end{proof}

\begin{remark}
Observe that Proposition~\ref{prop:bijectionFallingPaths} can be deduced from the following observations in the situation of double root free subword complexes studied in Section~\ref{sec:furtherCombinatorialProperties}. In this situation, the flip graph is the Hasse diagram of its transitive closure and the $\positiveEdgeLabel$- and $\negativeEdgeLabel$-labelings are both EL-labelings thereof.
By Theorem~\ref{thm:fallingPath}, all $\positiveEdgeLabel$- and $\negativeEdgeLabel$-falling paths have the same length. Therefore, Proposition~\ref{prop:Moebius} implies that they are equinumerous.
A similar topological construction in the situation of subword complexes having double roots is yet to be found.\footnote{We thank an anonymous referee for raising this question.}
\end{remark}


\subsection{Greedy facets}
\label{subsec:greedyFacets}

We now characterize the unique source and sink of the increasing flip graph~$\flipGraph$.

\begin{proposition}
\label{prop:increasingFlipGraph}
The lexicographically smallest (resp.~largest) facet of~$\subwordComplex$ is the unique source (resp.~sink) of~$\flipGraph$.
\end{proposition}

\begin{proof}
The lexicographically smallest facet is a source of~$\flipGraph$ since none of its flips can be decreasing.
We prove that this source is unique by induction on the word~$\Q$.
Denote by~$\positiveFacet[\Q_\eraseLast, \rho]$ (resp.~$\positiveFacet[\Q_\eraseLast, \rho q_m]$) the lexicographically smallest facet of~$\subwordComplex[\Q_\eraseLast, \rho]$ (resp.~$\subwordComplex[\Q_\eraseLast, \rho q_m]$) and assume that it is the unique source of the flip graph~$\flipGraph[\Q_\eraseLast, \rho]$ (resp.~$\flipGraph[\Q_\eraseLast, \rho q_m]$).
Consider a source~$\mathsf{P}$ of~$\flipGraph$.
We distinguish two cases:
\begin{enumerate}[$~\bullet$]
\item If~$\ell(\rho q_m) > \ell(\rho)$, then~$q_m$ cannot be the last reflection of a reduced expression for~$\rho$.
Thus~$\subwordComplex = \subwordComplex[\Q_\eraseLast, \rho] \join m$ and~$\mathsf{P} = \positiveFacet[\Q_\eraseLast, \rho] \cup m$.
\item Otherwise,~$\ell(\rho q_m) < \ell(\rho)$.
If~$m$ is in~$\mathsf{P}$, then
$$\Root{\mathsf{P}}{m} = \rho(\alpha_{q_m}) \in \Phi^-\cap\rho(\Phi^+).$$
Since $\Phi^-\cap\rho(\Phi^+) = -\inv(\rho)$, we obtain that $m$ is flippable (by Proposition~\ref{prop:roots&flips}\eqref{prop:roots&flips:flip}) and its flip is decreasing (by Proposition~\ref{prop:roots&flips}\eqref{prop:roots&flips:increasing}).
This would contradict the assumption that~$\mathsf{P}$ is a source of~$\flipGraph$.
Consequently,~$m \notin \mathsf{P}$.
Since the facets of~$\subwordComplex$ which do not contain~$m$ coincide with the facets of~$\subwordComplex[\Q_\eraseLast, \rho q_m]$, we obtain that $\mathsf{P} = \positiveFacet[\Q_\eraseLast,\rho q_m]$.
\end{enumerate}
In both cases, we obtain that the source~$\mathsf{P}$ is the lexicographically smallest facet of~$\subwordComplex$.
The proof is similar for the sink.
\end{proof}

We call \defn{positive} (resp.~\defn{negative}) \defn{greedy facet} and denote by $\positiveFacet$ (resp. $\negativeFacet$) the unique source (resp.~sink) of the graph~$\flipGraph$ of increasing flips.
The term ``positive'' (resp.~``negative'') emphasizes the fact that~$\positiveFacet$ (resp.~$\negativeFacet$) is the unique facet of~$\subwordComplex$ whose root configuration is a subset of positive (resp.~negative) roots, while the term ``greedy'' refers to the greedy properties of these facets underlined in Lemmas~\ref{lem:greedy1} and~\ref{lem:greedy2}.

\enlargethispage{.2cm}
These greedy facets are reverse to one another (see Remark~\ref{rem:reversal}).
Namely,
$$\negativeFacet[\q_m \cdots \q_1,\rho^{-1}] = \set{m+1-p}{p \in \positiveFacet[\q_1 \cdots \q_m, \rho]}.$$
We still work with both in parallel to simplify the presentation in the next section.

\begin{example}
Consider the subword complex $\subwordComplex[\Qex,\rhoex]$ presented in Example~\ref{exm:toto}. Its positive and negative greedy facets are $\positiveFacet[\Qex,\rhoex] = \{1,2,3,5,6\}$ and $\negativeFacet[\Qex,\rhoex] = \{3,4,7,8,9\}$, respectively, see \fref{fig:greedy}.
They appear respectively as the leftmost and rightmost facets in \fref{fig:flipGraph}.

\begin{figure}[ht]
  \centerline{\includegraphics[width=.9\textwidth]{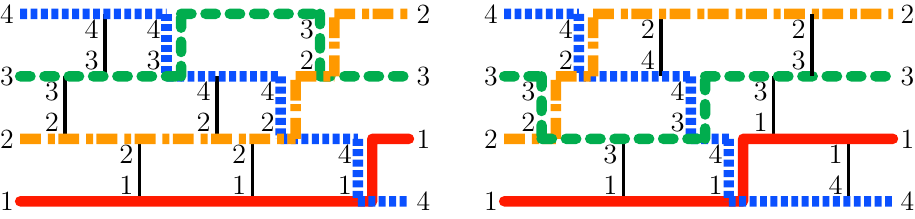}}
  \caption{The positive and negative greedy facets of~$\subwordComplex[\Qex,\rhoex]$.}
  \label{fig:greedy}
\end{figure}
\end{example}

The following two lemmas provide two (somehow inverse) greedy inductive procedures to construct the greedy facets~$\positiveFacet$ and~$\negativeFacet$.
These lemmas are direct consequences of the definition of the greedy facets and of the induction formulas for the facets~$\facets$ presented in Section~\ref{subsec:induction}.
Remember that we denote by~$\Q_\eraseFirst \eqdef \q_2 \cdots \q_m$ and~$\Q_\eraseLast \eqdef \q_1 \cdots \q_{m-1}$ the words on~$S$ obtained from~$\Q\! \eqdef \q_1 \cdots \q_m$ by deleting its first and last letters respectively, and by~$\shiftRight{X} \eqdef \set{x+1}{x \in X}$ the right shift of a subset~$X \subset \Z$.

\begin{lemma}
\label{lem:greedy1}
The greedy facets~$\positiveFacet$ and~$\negativeFacet$ can be constructed inductively from $\positiveFacet[\varepsilon,e] = \negativeFacet[\varepsilon,e] = \varnothing$ using the following formulas:
\begin{align*}
\positiveFacet & = \begin{cases} \positiveFacet[\Q_\eraseLast, \rho] \cup m & \text{if } m \text{ appears in all facets of } \subwordComplex, \\ \positiveFacet[\Q_\eraseLast, \rho q_m] & \text{otherwise.}\end{cases}
\hspace{1.2cm}
\\
\negativeFacet & = \begin{cases} 1 \cup \shiftRight{\negativeFacet[\Q_\eraseFirst, \rho]} & \text{if } 1 \text{ appears in all facets of } \subwordComplex, \\ \shiftRight{\negativeFacet[\Q_\eraseFirst, q_1\rho]} & \text{otherwise.}\end{cases}
\end{align*}
\end{lemma}

\begin{lemma}
\label{lem:greedy2}
The greedy facets~$\positiveFacet$ and~$\negativeFacet$ can be constructed inductively from $\positiveFacet[\varepsilon,e] = \negativeFacet[\varepsilon,e] = \varnothing$ using the following formulas:
\begin{align*}
\hspace{.2cm}
\positiveFacet & = \begin{cases} \shiftRight{\positiveFacet[\Q_\eraseFirst,q_1\rho]} & \text{if } 1 \text{ appears in none of the facets of } \subwordComplex, \\ 1 \cup \shiftRight{\positiveFacet[\Q_\eraseFirst,\rho]} & \text{otherwise.} \end{cases}
\\
\negativeFacet & = \begin{cases} \negativeFacet[\Q_\eraseLast,\rho q_m] & \text{if } m \text{ appears in none of the facets of } \subwordComplex, \\ \negativeFacet[\Q_\eraseLast,\rho] \cup m & \text{otherwise.} \end{cases}
\end{align*}
\end{lemma}

Lemmas~\ref{lem:greedy1} and~\ref{lem:greedy2} can be reformulated to obtain greedy sweep procedures on the word~$\Q$ itself, avoiding the use of induction.
Namely, the positive greedy facet~$\positiveFacet$ is obtained: 
\begin{enumerate}
\item either sweeping~$\Q$ from right to left placing inversions as soon as possible,
\item or sweeping~$\Q$ from left to right placing non-inversions as long as possible.
\end{enumerate}
The negative greedy facet is obtained similarly, reversing the directions of the sweeps.

\newpage
We have seen in Theorem~\ref{thm:ELlabeling} that for any two facets~$I,J \in \facets$ such that ${I \directedPath J}$, there is a $\positiveEdgeLabel$-rising (resp.~$\negativeEdgeLabel$-rising) path from~$I$ to~$J$.
In particular, there is always a $\positiveEdgeLabel$-rising (resp.~$\negativeEdgeLabel$-rising) path from~$\positiveFacet$ to~$\negativeFacet$.
We will now show that there is also at least one $\positiveEdgeLabel$-falling (resp.~$\negativeEdgeLabel$-falling) path from~$\positiveFacet$ to~$\negativeFacet$ if the subword complex~$\subwordComplex$ is spherical.

\begin{proposition}
\label{prop:fallingPath}
For any spherical subword complex~$\subwordComplex$, there is always a $\positiveEdgeLabel$-falling and an $\negativeEdgeLabel$-falling path from~$\positiveFacet$ to~$\negativeFacet$.
\end{proposition}

\begin{proof}
Since the subword complex~$\subwordComplex$ is spherical, recall that any position in any facet of~$\subwordComplex$ is flippable.
We will prove that starting from the positive greedy facet~$\positiveFacet$ and successively flipping all its positions in decreasing order yields the negative greedy facet~$\negativeFacet$, thus providing a $\positiveEdgeLabel$-falling path from~$\positiveFacet$ to~$\negativeFacet$.

Let~$\ell \eqdef |Q|-\ell(\rho)$ denote the size of each facet of~$\subwordComplex$.
Let~${\positiveEdgeLabel_1 > \dots > \positiveEdgeLabel_\ell}$ denote the positions of the positive greedy facet~$\positiveFacet$ in decreasing order.
We consider the $\positiveEdgeLabel$-falling path ${\positiveFacet = I_1 \edge \cdots \edge I_{\ell+1}}$ defined by $\positiveEdgeLabel(I_k \sep I_{k+1}) = \positiveEdgeLabel_k$.
We also set $\negativeEdgeLabel_k \eqdef \negativeEdgeLabel(I_k \sep I_{k+1})$.
By definition, we have ${I_k = \{\negativeEdgeLabel_1,\dots,\negativeEdgeLabel_{k-1},\positiveEdgeLabel_k,\dots,\positiveEdgeLabel_\ell \}}$.
We will prove that the root~$\Root{I_k}{\negativeEdgeLabel_j}$ is negative for any~$j < k \in [\ell+1]$.
This implies in particular that~$I_{\ell+1}$ is the negative greedy facet~$\negativeFacet$.

\newcommand{\x}{\mathsf{x}}
To see this, fix~$j \in [\ell]$. For any~$k \in [j+1,\ell+1]$, denote by~$\x_k$ the position in the complement of~$I_k$ such that~$\Root{I_k}{\x_k} = \pm \Root{I_k}{\negativeEdgeLabel_j}$.
We prove by induction on~$k$ that~$\positiveEdgeLabel_k < \x_k < \negativeEdgeLabel_j$, and thus (by Proposition~\ref{prop:roots&flips}\eqref{prop:roots&flips:increasing}) that the root~$\Root{I_k}{\negativeEdgeLabel_j} = -\Root{I_k}{\x_k}$ is negative for any~$j < k \le \ell + 1$.
First, this is immediate for~${k = j+1}$ since~$\x_{j+1} = \positiveEdgeLabel_j$ (because we just flipped out~$\positiveEdgeLabel_j$ to flip in~$\negativeEdgeLabel_j$ in~$I_j$) and ${\positiveEdgeLabel_{j+1} < \positiveEdgeLabel_j < \negativeEdgeLabel_j}$.
Assume now that we proved that~$\positiveEdgeLabel_k < \x_k < \negativeEdgeLabel_j$ for a certain~$k$.
We distinguish two cases:
\begin{enumerate}[(i)]
\item If~${\negativeEdgeLabel_k < \negativeEdgeLabel_j}$, then $\Root{I_{k+1}}{\negativeEdgeLabel_j} = \Root{I_{k}}{\negativeEdgeLabel_j}$ by Proposition~\ref{prop:roots&flips}\eqref{prop:roots&flips:update}.
Since this root is negative, Proposition~\ref{prop:roots&flips}\eqref{prop:roots&flips:increasing} ensures that~$\x_{k+1} < \negativeEdgeLabel_j$.
Moreover, if~${\x_{k+1} \le \positiveEdgeLabel_{k+1}}$, then we would have~$\x_{k+1} < \positiveEdgeLabel_k$, and thus Proposition~\ref{prop:roots&flips}\eqref{prop:roots&flips:update} would give
$$\Root{I_k}{\x_{k+1}} = \Root{I_{k+1}}{\x_{k+1}} = -\Root{I_{k+1}}{\negativeEdgeLabel_j} = -\Root{I_{k}}{\negativeEdgeLabel_j}.$$
By definition, this would imply that~$\x_k = \x_{k+1} < \positiveEdgeLabel_k$, contradicting the induction hypothesis.
\item If~${\negativeEdgeLabel_k > \negativeEdgeLabel_j}$, then we have~$\positiveEdgeLabel_k < \x_k < \negativeEdgeLabel_j < \negativeEdgeLabel_k$. Therefore, Proposition~\ref{prop:roots&flips}\eqref{prop:roots&flips:update} ensures that
$$\Root{I_{k+1}}{\x_k} = s_{\Root{I_k}{\positiveEdgeLabel_k}}(\Root{I_k}{\x_k}) = -s_{\Root{I_k}{\positiveEdgeLabel_k}}(\Root{I_k}{\negativeEdgeLabel_j}) = -\Root{I_{k+1}}{\negativeEdgeLabel_j}.$$
By definition, this implies that~$\x_{k+1} = \x_k$.
\end{enumerate}
In both cases, we obtained that~$\positiveEdgeLabel_{k+1} < \x_{k+1} < \negativeEdgeLabel_j$, thus concluding our inductive argument.

The proof for the $\negativeEdgeLabel$-falling path is similar.
\end{proof}

Note that this proposition fails if we drop the condition that~$\subwordComplex$ is spherical, as illustrated in the subword complex~$\subwordComplex[\Qex,\rhoex]$ of Example~\ref{exm:toto}.
A smaller example is given by the subword complex~$\subwordComplex[\sq{\tau}_1\sq{\tau}_2\sq{\tau}_1\sq{\tau}_2, \tau_1\tau_2]$.


\subsection{Spanning trees}
\label{subsec:spanningTrees}

As discussed in Remark~\ref{rem:spanningTrees}, the edge labelings~$\positiveEdgeLabel$ and~$\negativeEdgeLabel$ automatically produce canonical spanning trees of any interval of the increasing flip graph~$\flipGraph$.
Since~$\flipGraph$ has a unique source~$\positiveFacet$ and a unique sink~$\negativeFacet$, we obtain in particular four spanning trees of the graph~$\flipGraph$ itself.
The goal of this section is to give alternative descriptions of these four spanning trees.

We call respectively \defn{positive source tree}, \defn{positive sink tree}, \defn{negative source tree}, and \defn{negative sink tree}, and denote respectively by~$\positiveSourceTree$, $\positiveSinkTree$, $\negativeSourceTree$, and~$\negativeSinkTree$, the $\positiveEdgeLabel$-source, $\positiveEdgeLabel$-sink, $\negativeEdgeLabel$-source, and $\negativeEdgeLabel$-sink trees of~$\flipGraph$.
The tree~$\positiveSourceTree$ (resp.~$\negativeSourceTree$) is formed by all $\positiveEdgeLabel$-rising (resp.~$\negativeEdgeLabel$-rising) paths from the positive greedy facet~$\positiveFacet$ to all the facets of~$\subwordComplex$.
Both~$\positiveSourceTree$ and~$\negativeSourceTree$ are rooted at and directed away from the positive greedy facet~$\positiveFacet$.
The tree~$\positiveSinkTree$ (resp.~$\negativeSinkTree$) is formed by all $\positiveEdgeLabel$-rising (resp.~$\negativeEdgeLabel$-rising) paths from all the facets of~$\subwordComplex$ to the negative greedy facet~$\negativeFacet$.
Both~$\positiveSinkTree$ and $\negativeSinkTree$ are rooted at and directed towards the negative greedy facet~$\negativeFacet$.

The positive source and negative sink trees (resp.~the positive sink and the negative source trees) are reverse to one another (see Remark~\ref{rem:reversal}).
Namely, as we already observed, $I \edge J$ is an edge in the increasing flip graph $\flipGraph[\q_m \cdots \q_1, \rho^{-1}]$ if and only if $J' \eqdef \set{m+1-j}{j \in J} \edge I' \eqdef \set{m+1-i}{i \in I}$ is an edge in the increasing flip graph $\flipGraph[\q_1 \cdots \q_m, \rho]$.
Moreover, $I \edge J$ belongs to~$\positiveSourceTree[\q_m \cdots \q_1, \rho^{-1}]$ if and only if $J' \edge I'$ belongs to~$\negativeSinkTree[\q_1 \cdots \q_m, \rho]$.
Similarly, $I \edge J$ belongs to~$\positiveSinkTree[\q_m \cdots \q_1, \rho^{-1}]$ if and only if $J' \edge I'$ belongs to~$\negativeSourceTree[\q_1 \cdots \q_m, \rho]$.

\begin{example}
\label{exm:totoTrees}
Consider the subword complex~$\subwordComplex[\Qex,\rhoex]$ from Example~\ref{exm:toto}.
Figures~\ref{fig:positiveSourceTree}, \ref{fig:positiveSinkTree}, \ref{fig:negativeSourceTree}, and~\ref{fig:negativeSinkTree} represent respectively the trees~$\positiveSourceTree[\Qex,\rhoex]$, $\positiveSinkTree[\Qex,\rhoex]$, $\negativeSourceTree[\Qex,\rhoex]$, and~$\negativeSinkTree[\Qex,\rhoex]$.
Observe that these four canonical spanning trees of~$\flipGraph$ are all different in general.
\begin{figure}[p]
  \centerline{\includegraphics[width=\textwidth]{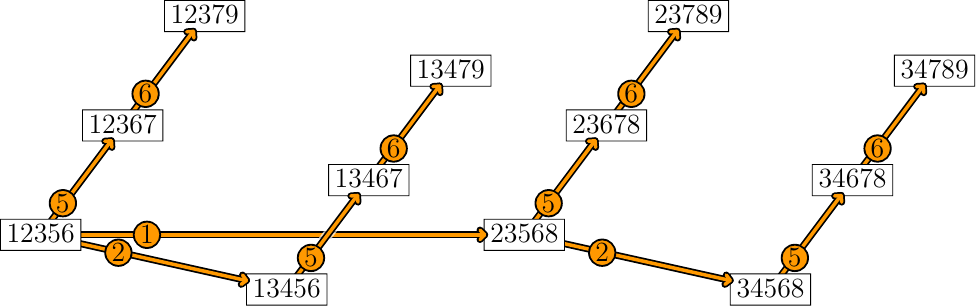}}
  \caption{The positive source tree~$\positiveSourceTree[\Qex,\rhoex]$.}
  \label{fig:positiveSourceTree}
\end{figure}
\begin{figure}[p]
  \centerline{\includegraphics[width=\textwidth]{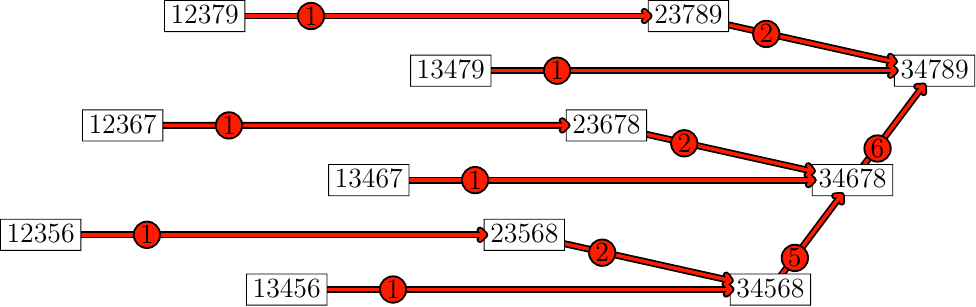}}
  \caption{The positive sink tree~$\positiveSinkTree[\Qex,\rhoex]$.}
  \label{fig:positiveSinkTree}
\end{figure}
\begin{figure}[p]
  \centerline{\includegraphics[width=\textwidth]{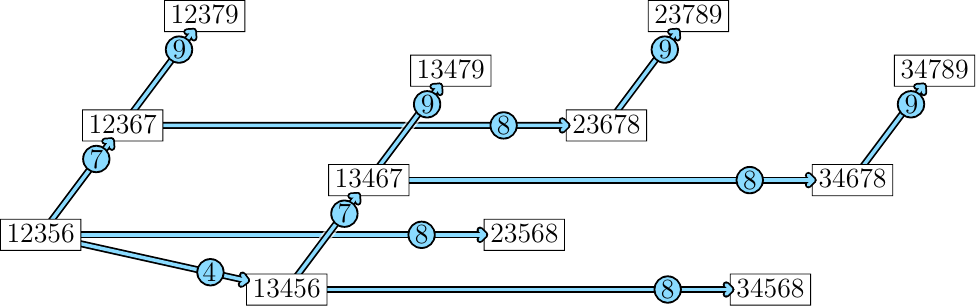}}
  \caption{The negative source tree~$\negativeSourceTree[\Qex,\rhoex]$.}
  \label{fig:negativeSourceTree}
\end{figure}
\begin{figure}[p]
  \centerline{\includegraphics[width=\textwidth]{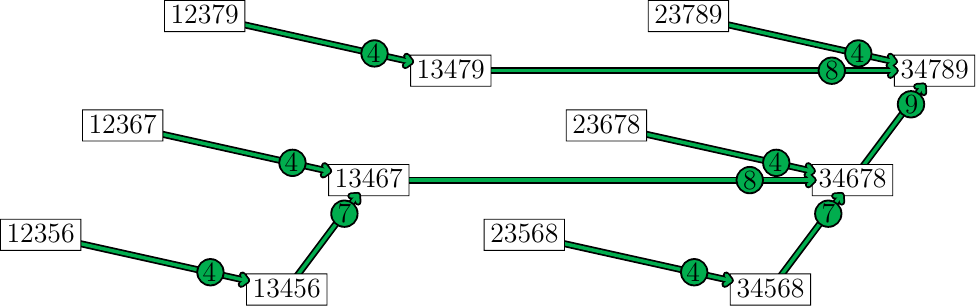}}
  \caption{The negative sink tree~$\negativeSinkTree[\Qex,\rhoex]$.}
  \label{fig:negativeSinkTree}
\end{figure}
\end{example}

We now give a direct description of the father of a facet~$I$ in~$\positiveSinkTree$ and~$\negativeSourceTree$ in terms of~$I \ssm \negativeFacet$ and~$I \ssm \positiveFacet$.

\begin{proposition}
Let~$I$ be a facet of~$\subwordComplex$.
If~$I \ne \negativeFacet$, then the father of~$I$ in~$\positiveSinkTree$ is obtained from~$I$ by flipping the smallest position in~$I \ssm \negativeFacet$.
Similarly, if~$I \ne \positiveFacet$, then the father of~$I$ in~$\negativeSourceTree$ is obtained from~$I$ by flipping the largest position in~$I \ssm \positiveFacet$.
\end{proposition}

\begin{proof}
Since the father of~$I$ in~$\positiveSinkTree$ (resp.~in~$\negativeSourceTree$) is the facet next to~$I$ on the unique $\positiveEdgeLabel$-rising path towards~$\negativeFacet$ (resp.~the facet previous to~$I$ on the unique $\negativeEdgeLabel$-rising path from~$\positiveFacet$), this is a direct consequence of Proposition~\ref{prop:unicity}.
\end{proof}

We now focus on the positive source tree~$\positiveSourceTree$ and on the negative sink tree $\negativeSinkTree$, and provide two different descriptions of them.
The first is an inductive description of~$\positiveSourceTree$ and~$\negativeSinkTree$ (see Propositions~\ref{prop:inductiveNegativeTree} and~\ref{prop:inductivePositiveTree}).
The second is a direct description of the father of a facet~$I$ in~$\positiveSourceTree$ and~$\negativeSinkTree$ in terms of greedy prefixes and suffixes of~$I$ (see Propositions~\ref{prop:characterizationFatherNegative} and~\ref{prop:characterizationFatherPositive}).
These descriptions mainly rely on the following property of the greedy facets.

\begin{proposition}
\label{prop:gfp}
If~$m$ is a flippable position of~$\negativeFacet$, then~$\negativeFacet[\Q_\eraseLast,\rho q_m]$ is obtained from~$\negativeFacet$ by flipping~$m$.
Similarly, if~$1$ is a flippable position of~$\positiveFacet$, then~$\positiveFacet[\Q_\eraseFirst,q_1\rho]$ is obtained from~$\positiveFacet$ by flipping~$1$ and shifting to the left.
\end{proposition}

\begin{proof}
Although the formulation is simpler for the negative greedy facets, the proof is simpler for the positive ones (due to the direction chosen in the definition of the root function).
Assume that~$1$ is a flippable position of~$\positiveFacet$.
Let~$J \in \facets$ and~$j \in J$ be such that $\positiveFacet \ssm 1 = J \ssm j$.
Consider the facet~$\shiftLeft{J}$ of~$\subwordComplex[\Q_\eraseFirst,q_1\rho]$ obtained shifting~$J$ to the left.
Proposition~\ref{prop:roots&flips}\eqref{prop:roots&flips:update} enables us to compute the root function~$\Root{J}{\cdot}$ for~$J$, which in turn gives us the root function for~$\shiftLeft{J}$:
$$\Root{\shiftLeft{J}}{k} = \begin{cases} \Root{\positiveFacet}{k+1} & \text{if } 1 \le k \le j-1, \\ q_1(\Root{\positiveFacet}{k+1}) & \text{otherwise.} \end{cases}$$
Since all positions~$i \in \positiveFacet$ such that~$\Root{\positiveFacet}{i} = \alpha_{q_1}$ are located before~$j$, and since $\alpha_{q_1}$ is the only positive root sent to a negative root by the simple reflection~$q_1$, all roots~$\Root{\shiftLeft{J}}{k}$, for~$k \in \shiftLeft{J}$, are positive.
Consequently,~$\shiftLeft{J} = \positiveFacet[\Q_\eraseFirst,q_1\rho]$.

We obtain the result for negative facets using the reversal operation of Remark~\ref{rem:reversal}.
\end{proof}

\begin{example}
Consider the subword complex~$\subwordComplex[\Qex,\rhoex]$ of Example~\ref{exm:toto}.
Since position $9$ is flippable in~$\negativeFacet[\Qex,\rhoex] = \{3,4,7,8,9\}$, we have $\negativeFacet[\Qex_\eraseLast,\rhoex\tau_1] = \{3,4,6,7,8\}$.
Moreover, since position~$1$~is flippable in~$\positiveFacet[\Qex,\rhoex] = \{1,2,3,5,6\}$, we have $\positiveFacet[\Qex_\eraseFirst,\tau_2\rhoex] = \shiftLeft{\{2,3,5,6,8\}} = \{1,2,4,5,7\}$.
\end{example}

Using Proposition~\ref{prop:gfp}, we can describe inductively the two trees~$\positiveSourceTree$ and $\negativeSinkTree$.
The induction follows the induction formulas for the facets~$\facets$ presented in Section~\ref{subsec:induction}.
For a tree~$\cT$ whose vertices are subsets of~$\Z$ and for an element~$z \in \Z$, we denote by~$\cT \join z = z \join \cT$ the tree with a vertex~$X \cup z$ for each vertex~$X$ of~$\cT$ and an edge~$X \cup z \edge Y \cup z$ for each edge~$X \edge Y$ of~$\cT$.
Similarly, we will denote by~$\shiftRight{\cT}$ the tree with a vertex~$\shiftRight{X} \eqdef \set{x+1}{x \in X}$ for each vertex~$X$ of~$\cT$ and an edge~$\shiftRight{X} \edge \shiftRight{Y}$ for each edge~$X \edge Y$ of~$\cT$.

We start with the inductive description of the negative sink tree~$\negativeSinkTree$, which is based on the right induction formula.
For the empty word~$\varepsilon$, the tree~$\negativeSinkTree[\varepsilon,e]$ is formed by the unique facet~$\varnothing$ of~$\subwordComplex[\varepsilon,e]$, and the tree~$\negativeSinkTree[\varepsilon,\rho]$ is empty if~$\rho \ne e$.
Otherwise, $\negativeSinkTree$ is obtained as follows.

\begin{proposition}
\label{prop:inductiveNegativeTree}
For a non-empty word~$\Q$, the tree~$\negativeSinkTree$ equals
\begin{enumerate}[(i)]
\item $\negativeSinkTree[\Q_\eraseLast, \rho q_m]$ if~$m$ appears in none of the facets of~$\subwordComplex$; \label{prop:inductiveNegativeTree:i}
\item $\negativeSinkTree[\Q_\eraseLast, \rho] \join m$ if~$m$ appears in all the facets of~$\subwordComplex$; \label{prop:inductiveNegativeTree:ii}
\item the disjoint union of~$\negativeSinkTree[\Q_\eraseLast, \rho q_m]$ and~$\negativeSinkTree[\Q_\eraseLast, \rho] \join m$, with an additional edge from~$\negativeFacet[\Q_\eraseLast, \rho q_m]$ to~$\negativeFacet = \negativeFacet[\Q_\eraseLast, \rho] \cup m$, otherwise. \label{prop:inductiveNegativeTree:iii}
\end{enumerate}
\end{proposition}

\begin{proof}
Assume that~$m$ is contained in at least one and not all facets of~$\subwordComplex$.
In other words, $m$ is a flippable position of~$\negativeFacet$.
Let~$I = I_1 \edge \cdots \edge I_{\ell+1} = \negativeFacet$ be any $\negativeEdgeLabel$-rising path from an arbitrary facet~$I \in \facets$ to~$\negativeFacet$.
If the label~$m$ appears in this path, then it should clearly appear last.
By Proposition~\ref{prop:gfp}, we have therefore ${I_\ell = \negativeFacet[\Q_\eraseLast, \rho q_m]}$, and $I = I_1 \edge \cdots \edge I_\ell = \negativeFacet[\Q_\eraseLast, \rho q_m]$ is also an $\negativeEdgeLabel$-rising path from~$I$ to~$\negativeFacet[\Q_\eraseLast, \rho q_m]$ in the increasing flip graph~$\flipGraph[\Q_\eraseLast, \rho q_m]$.
Other\-wise, if the label~$m$ does not appear in the path, then~$m$ is contained in all facets of this path, and~$(I \ssm m) = (I_1 \ssm m) \edge \cdots \edge (I_{\ell+1} \ssm m) = \negativeFacet[\Q_\eraseLast, \rho]$ is an $\negativeEdgeLabel$-rising path from~$I \ssm m$ to~$\negativeFacet[\Q_\eraseLast, \rho]$ in the increasing flip graph~$\flipGraph[\Q_\eraseLast, \rho]$.
This corresponds precisely to the description of~\eqref{prop:inductiveNegativeTree:iii}.
The proofs of~\eqref{prop:inductiveNegativeTree:i} and~\eqref{prop:inductiveNegativeTree:ii} are similar and left to the reader.
\end{proof}

We now give the inductive description of the positive source tree~$\positiveSourceTree$, which is based on the left induction formula.
For the empty word~$\varepsilon$, the tree~$\positiveSourceTree[\varepsilon,e]$ is formed by the unique facet~$\varnothing$ of~$\subwordComplex[\varepsilon,e]$, and the tree~$\positiveSourceTree[\varepsilon,\rho]$ is empty if~$\rho \ne e$.
Otherwise, $\positiveSourceTree$ is obtained as follows.

\begin{proposition}
\label{prop:inductivePositiveTree}
For a non-empty word~$\Q$, the tree~$\positiveSourceTree$ equals
\begin{enumerate}[(i)]
\item $\shiftRight{\positiveSourceTree[\Q_\eraseFirst, q_1 \rho]}$ if~$1$ appears in none of the facets of~$\subwordComplex$;
\item $1 \join \shiftRight{\positiveSourceTree[\Q_\eraseFirst, \rho]}$ if~$1$ appears in all the facets of~$\subwordComplex$;
\item the disjoint union of~$\shiftRight{\positiveSourceTree[\Q_\eraseFirst, q_1 \rho]}$ and~$1 \join \shiftRight{\positiveSourceTree[\Q_\eraseFirst, \rho]}$, with an additional edge from~$\positiveFacet = 1 \cup \shiftRight{\positiveFacet[\Q_\eraseFirst, \rho]}$ to~$\shiftRight{\positiveFacet[\Q_\eraseFirst, q_1\rho]}$, otherwise.
\end{enumerate}
\end{proposition}

\begin{proof}
We can either translate the proof of Proposition~\ref{prop:inductiveNegativeTree}, or directly apply to Proposition~\ref{prop:inductiveNegativeTree} the reversal operation of Remark~\ref{rem:reversal}.
\end{proof}

Note that we do not have a similar inductive description for the positive sink and negative source trees~$\positiveSinkTree$ and~$\negativeSourceTree$.
Let~$I_{\max}$ denote the neighbor of~$\negativeFacet$ in~$\flipGraph$ which maximizes~$\positiveEdgeLabel_{\max} \eqdef \positiveEdgeLabel(I_{\max} \sep \negativeFacet)$.
We can use position~$\positiveEdgeLabel_{\max}$ to decompose the positive sink tree~$\positiveSinkTree$ as the union of a spanning tree of the graph of increasing flips on its link~$\set{I \in \subwordComplex}{\positiveEdgeLabel_{\max} \in I}$ with a spanning tree of the graph of increasing flips on its deletion~$\set{I \in \subwordComplex}{\positiveEdgeLabel_{\max} \notin I}$, together with the edge~$I_{\max} \sep \negativeFacet$.
However, contrarily to the link of~$\positiveEdgeLabel_{\max}$, the deletion of~$\positiveEdgeLabel_{\max}$ is not a subword complex in general.
This is a serious limit to an inductive decomposition of the positive sink tree~$\positiveSinkTree$.
The same observation holds for the negative source tree~$\negativeSourceTree$.

\medskip
We now give a direct characterization of the father of a facet~$I$ of~$\subwordComplex$ in the negative sink tree~$\negativeSourceTree$.
This description can be understood in terms of the longest greedy prefix of~$I$.

\begin{proposition}
\label{prop:characterizationFatherNegative}
Let~$I \ne \negativeFacet$ be a facet of~$\subwordComplex$.
Define $y = y(I)$ to be the smallest position in~$[m]$ such that
$$I \cap [y] \ne \negativeFacet[\q_1 \cdots \q_y, \wordprod{\Q}{[y] \ssm I}],$$
and $x = x(I)$ to be the smallest position in~$I$ such that~$\Root{I}{x} = \Root{I}{y}$.
Then the father of the facet~$I$ in the negative sink tree~$\negativeSinkTree$ is obtained from~$I$ by flipping~$x$.
\end{proposition}

\begin{proof}
Let~$x(I)$ and~$y(I)$ be the positions defined in the statement of the proposition.
Denote by~$J$ the father of~$I$ in the negative sink tree~$\negativeSinkTree$, and let~$\bar x(I)$ and~$\bar y(I)$ be such that $I \ssm \bar x(I) = J \ssm \bar y(I)$.
We want to prove that~$x(I) = \bar x(I)$ and $y(I) = \bar y(I)$ for any facet~$I \ne \negativeFacet$ of~$\subwordComplex$.

We first prove that~$y(I) = \bar y(I)$ for any facet~$I$ of~$\subwordComplex$ by induction on the negative sink tree.
For this, set~$y(\negativeFacet) = \bar y(\negativeFacet) = m+1$.
Consider an arbitrary facet~$I \ne \negativeFacet$ and its father~$J$ in~$\negativeSinkTree$.
In particular, we have $I \ssm \bar x(I) = J \ssm \bar y(I)$ with~$\bar x(I) < \bar y(I) < \bar y(J)$.
The first inequality holds since the flip~$I \edge J$ is increasing, and the second holds since the unique path from~$I$ to~$\negativeFacet$ in~$\negativeSinkTree$ is $\negativeEdgeLabel$-rising.
We want to prove that~${y(I) = \bar y(I)}$, assuming by induction that~$y(J) = \bar y(J)$.
First, since~$\bar y(I) < \bar y(J) = y(J)$ and~${\wordprod{\Q}{[\bar y(I)] \ssm J} = \wordprod{\Q}{[\bar y(I)] \ssm I}}$, we observe that
$${\bar y(I) \in J \cap [\bar y(I)] = \negativeFacet[\q_1 \cdots \q_{\bar y(I)},\wordprod{\Q}{[\bar y(I)] \ssm J}] = \negativeFacet[\q_1 \cdots \q_{\bar y(I)},\wordprod{\Q}{[\bar y(I)] \ssm I}]}.$$
Since $\bar y(I) \notin I \cap [\bar y(I)]$, this implies that~$y(I) \le \bar y(I)$.
Second, the negative greedy flip property of Proposition~\ref{prop:gfp} ensures that
$$I \cap [\bar y(I)-1] = \negativeFacet[\q_1 \cdots \q_{\bar y(I)-1}, \wordprod{\Q}{[\bar y(I)-1] \ssm I}]$$
since it is obtained from~$J \cap [\bar y(I)] = \negativeFacet[\q_1 \cdots \q_{\bar y(I)}, \wordprod{\Q}{[\bar y(I)] \ssm J}]$ by flipping~$\bar y(I)$.
Thus, we obtain that~$y(I) > \bar y(I)-1$.
This concludes the proof that~$y(I) = \bar y(I)$.

Finally, since~$I \ssm \bar x(I) = J \ssm \bar y(I) = J \ssm y(I)$, we know that~$\Root{I}{\bar x(I)} = \Root{I}{y(J)}$ by Proposition~\ref{prop:roots&flips}\eqref{prop:roots&flips:flip}. Moreover, it has to be the smallest position in~$I$ with this property since otherwise~$y(J)$ would be smaller than~$y(I)$.
\end{proof}

Finally, we give a similar direct characterization of the father of a facet~$I$ of~$\subwordComplex$ in the positive source tree~$\positiveSourceTree$.
This description can be understood in terms of the longest greedy suffix of~$I$.

\begin{proposition}
\label{prop:characterizationFatherPositive}
Let~$I \ne \positiveFacet$ be a facet of~$\subwordComplex$.
Define $y = y(I)$ to be the largest position in~$[m]$ such that
$${\set{i-y}{i \in I \ssm [y]} \ne \positiveFacet[\q_{y+1} \cdots \q_m, \wordprod{\Q}{[y+1,m] \ssm I}]},$$
and $x = x(I)$ to be the largest position in~$I$ such that~$\Root{I}{x} = -\Root{I}{y}$.
Then the father of the facet~$I$ in the positive sink tree~$\positiveSinkTree$ is obtained from~$I$ by flipping~$x$.
\end{proposition}

\begin{proof}
We can either translate the proof of Proposition~\ref{prop:characterizationFatherNegative}, or directly apply to Proposition~\ref{prop:characterizationFatherNegative} the reversal operation of Remark~\ref{rem:reversal}.
\end{proof}


\subsection{Greedy flip algorithm}
\label{subsec:greedyFlipAlgorithm}

The initial motivation of this paper was to find efficient algorithms for the exhaustive generation of the set~$\facets$ of facets of the subword complex~$\subwordComplex$.
For the evaluation of the time and space complexity of the different enumeration algorithms, we consider as parameters the rank~$n$ of the Coxeter group~$W$ and the size~$m$ of the word~$\Q$.
Neither of these two parameters can be considered to be constant a priori.
For example, if we want to generate all triangulations of a convex $(n+3)$-gon (see Example~\ref{exm:geometricGraphs}), we consider a subword complex with a group~$W$ of rank~$n$ and with a word~$\Q$ of size~${n(n+3)/2}$.

The properties of the subword complex described in Sections~\ref{subsec:induction} and~\ref{subsec:flips&roots} already provide two immediate enumeration algorithms.
First, the inductive structure of~$\facets$ yields an \defn{inductive algorithm} whose running time per facet is polynomial.
More precisely, since all subword complexes which appear in the different cases of the right induction formula of Section~\ref{subsec:induction} are non-empty, and since the tests~$\rho \not\prec \Q_\eraseLast$ and~$\ell(\rho q_m) > \ell(\rho)$ can be performed in~$O(mn)$ time, the running time per facet of this inductive algorithm is in~$O(m^2n)$.

The second option is an \defn{exploration of the flip graph}~$\flipGraph$.
This flip graph is connected by Theorem~\ref{theo:KnutsonMiller}, and it has degree bounded by~$m-\ell(\rho)$.
We can thus generate~$\facets$ exploring the flip graph, and we need~${O(m-\ell(\rho))}$ flips per facet for this exploration.
By Proposition~\ref{prop:roots&flips}, we can perform flips in the subword complex~$\subwordComplex$ in~$O(mn)$ time if we store and update the facets of~$\facets$ together with their root functions (note that this storage requires~$O(mn)$ space).
We thus obtain again a running time of $O(m^2n)$ per facet.
The problem of a naive exploration of the flip graph is that we need to store all facets of~$\facets$ during the algorithm, which may require an exponential working space.
This happens for example if we want to generate the~$\frac{1}{n+2}\binom{2n+2}{n+1}$ triangulations of a convex $(n+3)$-gon (see Example~\ref{exm:geometricGraphs}).

Using the canonical spanning trees constructed in this paper, we can bypass this difficulty: we avoid to store all visited facets while preserving the same running time.
The \defn{greedy flip algorithm} generates all facets of the subword complex~$\subwordComplex$ by a depth first search procedure on one\footnote{As observed by M.~Pocchiola, searching on the positive sink tree or on the negative source tree improves the working space of the algorithm. This issue is relevant for the enumeration of pseudotriangulations and will be discussed in a forthcoming paper of his.} of the four canonical spanning trees described in Section~\ref{subsec:spanningTrees}.
The preorder traversal of the tree also provides an iterator on the facets of~$\subwordComplex$.
Given a facet~$I \in \facets$, we can indeed compute its next element in the preorder traversal of the spanning tree, provided we know its root function (plus the path from~$I$ to the root in the tree if we work with either~$\positiveSourceTree$ or~$\negativeSinkTree$).
These data can be updated at each step of the algorithm, using Proposition~\ref{prop:roots&flips} for the root function.

We now bound the time and space complexity of the greedy flip algorithm.
First, its working space is in~$O(mn)$ since we only need to remember during the algorithm the current facet, together with its root function (plus its path to the root in the tree if we work with either~$\positiveSourceTree$ or~$\negativeSinkTree$).
Concerning running time, each facet needs at most~$m$ flips to generate all its children in the spanning tree.
Since a flip can be performed in~$O(mn)$ time (by Proposition~\ref{prop:roots&flips}), the running time per facet of the greedy flip algorithm is still in~$O(m^2n)$.

\begin{figure}
  \centerline{\includegraphics[width=.49\textwidth]{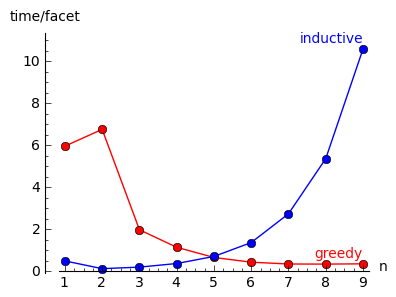} \quad \includegraphics[width=.49\textwidth]{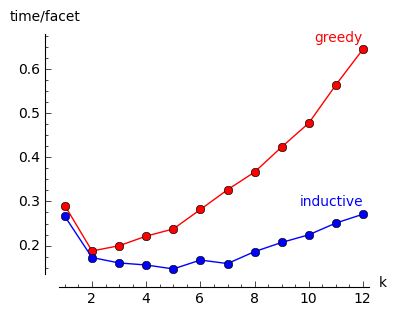}}
  \caption{Comparison of the running times of the inductive algorithm and the greedy flip algorithm to generate the $k$-cluster complex of type~$A_n$.
On the left, $k$ is fixed at~$1$ while~$n$ increases; on the right, $n$ is fixed at~$3$ while~$k$ increases.
The time is presented in millisecond per facet.}
  \label{fig:runnngTimes}
\end{figure}

We have implemented the greedy flip algorithm using the mathematical software Sage~\cite{sage} as part of a project\footnote{The ongoing work on this patch can be found at \url{http://trac.sagemath.org/sage_trac/ticket/11010}.} on implementing subword complexes.
We have seen that these two algorithms for generating facets have the same theoretical complexity, namely $O(m^2n)$ per facet.
To compare their experimental running time, we have constructed the $k$-cluster complex of type~$A_n$ for increasing values of~$k$ and~$n$.
Its facets correspond to the $k$-triangulations of the $(n+2k+1)$-gon (see Example~\ref{exm:geometricGraphs} and~\cite{CeballosLabbeStump} for the definition of multicluster complexes in any finite type).
The rank of the group is~$n$, while the length of the word is~$kn+\binom{n}{2}$.
\fref{fig:runnngTimes} presents the running time per facet for both enumeration algorithms in two situations: on the left, $k$ is fixed at~$1$ while~$n$ increases; on the right, $n$ is fixed at~$3$ while~$k$ increases.
The greedy flip algorithm is better than the inductive algorithm in the first situation, and worse in the second.
We observe a similar behavior for the computation of $k$-cluster complexes of types~$B_n$ and~$D_n$.
In general, the inductive algorithm is experimentally faster when the Coxeter group is fixed, but slower when the size of the Coxeter group increases.

\begin{remark}
Our algorithm is similar to that of~\cite{BronnimannKettnerPocchiolaSnoeying} for pointed triangulations and that of~\cite{PilaudPocchiola} for primitive sorting networks.
More precisely, the algorithms of~\cite{BronnimannKettnerPocchiolaSnoeying} and~\cite{PilaudPocchiola} are both depth first search procedures on the positive source tree of particular subword complexes: subword complexes modeling pointed pseudotriangulations for~\cite{BronnimannKettnerPocchiolaSnoeying} (see Example~\ref{exm:geometricGraphs}), and type~$A$ spherical subword complexes for~\cite{PilaudPocchiola}.
\end{remark}


\section{Further combinatorial properties of the EL-labelings}
\label{sec:furtherCombinatorialProperties}

In this section, we discuss some implications of the EL-labelings of the increasing flip graph presented in Section~\ref{subsec:ELlabelingsflipgraph}.
These results concern combinatorial properties of the \defn{increasing flip poset}~$\flipPoset$, defined as the transitive closure of the increasing flip graph~$\flipGraph$.
The key requirement for the validity of these results is that the increasing flip graph~$\flipGraph$ coincides with the Hasse diagram of the increasing flip poset~$\flipPoset$ (see the discussion in the beginning of Section~\ref{subsec:ELlabelings}).
We first characterize and study the subword complexes which fulfill this property.


\subsection{Double root free subword complexes}
\label{subsec:noDoubleRoot}

We say that the subword complex~$\subwordComplex$ has a \defn{double root} if there is a facet~$I \in \subwordComplex$ and two distinct positions~$i \ne j \in [m]$ both flippable in~$I$ such that~$\Root{I}{i} = \Root{I}{j}$.
Otherwise, we say that the subword complex~$\subwordComplex$ is \defn{double root free}.
In this section, we focus on double root free subword complexes due to the following characterization.

\begin{proposition}
\label{prop:characterizationDoubleRootFree}
The subword complex~$\subwordComplex$ is double root free if and only if its increasing flip graph~$\flipGraph$ coincides with the Hasse diagram of its increasing flip poset~$\flipPoset$.
\end{proposition}

\begin{proof}
Assume that~$\subwordComplex$ has a double root. Let~$i \ne j \in [m]$ be both flippable in~$I$, and let ${k \in [m] \ssm I}$ be such that~${\Root{I}{i} = \Root{I}{j} = \pm\Root{I}{k}}$ so that both~$i$ and~$j$ flip to~$k$.
Then the flip graph~$\flipGraph$ contains a triangle formed by the facets~$I$, $I \symdif \{i,k\}$, and~$I \symdif \{j,k\}$ (where $A \symdif B \eqdef (A \cup B) \ssm (A \cap B)$ denotes the \defn{symmetric difference} of two sets $A$ and $B$).
Since a Hasse diagram cannot contain a triangle, the Hasse diagram of the increasing flip poset~$\flipPoset$ is only a strict subgraph of the increasing flip graph~$\flipGraph$.

Assume reciprocally that the Hasse diagram of the increasing flip poset~$\flipPoset$ is a strict subgraph of the increasing flip graph~$\flipGraph$.
Let~$I \edge J$ be an oriented edge in~$\flipGraph$ which is not an edge in the Hasse diagram of~$\flipPoset$.
Let~${i \in I}$ and~${j \in J}$ be such that~$I \ssm i = J \ssm j$ (thus~$i < j$), and consider a path ${I = I_1 \edge \cdots \edge I_{\ell+1} = J}$ of increasing flips which prevents the edge~$I \edge J$ to be in the Hasse diagram of~$\flipPoset$ (in particular, $\ell > 1$).
Let ${\positiveEdgeLabel_1 > \ldots > \positiveEdgeLabel_\ell}$ be the decreasing reordering of the set~$\big\{\positiveEdgeLabel(I_1 \sep I_2),\ldots,\positiveEdgeLabel(I_{\ell} \sep I_{\ell+1})\big\}$ of positive edge labels along this path, and let $\negativeEdgeLabel_1,\ldots,\negativeEdgeLabel_\ell$ be the corresponding negative edge labels.
That is to say, when we flip~$\positiveEdgeLabel_k$ out of a certain facet in this path, we obtain~$\negativeEdgeLabel_k$ in the next facet of the path.
Since $I$ and $J$ differ only in positions $i$ and $j$ with $i<j$, and all flips are increasing, no position smaller than $i$ can be flipped.
Thus, we obtain that $\positiveEdgeLabel_{\ell} = i$, and by a similar argument that $\negativeEdgeLabel_1 = j$.
Applying the same argument to the other positions that are flipped along the path, in increasing or in decreasing order, moreover gives
$$i = \positiveEdgeLabel_{\ell} < \negativeEdgeLabel_\ell = \positiveEdgeLabel_{\ell-1} < \dots < \negativeEdgeLabel_2 = \positiveEdgeLabel_1 < \negativeEdgeLabel_1 = j.$$
Proposition~\ref{prop:roots&flips} thus ensures that all roots $\Root{I}{\positiveEdgeLabel_1},\dots,\Root{I}{\positiveEdgeLabel_\ell}$ coincide and are equal to $\Root{I}{\negativeEdgeLabel_1}$, and that we moreover have $\positiveEdgeLabel_k = \positiveEdgeLabel(I_k \sep I_{k+1})$ and~$\negativeEdgeLabel_k = \negativeEdgeLabel(I_k \sep I_{k+1})$.
Since~$\ell > 1$, this completes the proof.
\end{proof}

The intervals in the increasing flip graph of a double root free subword complex have the following property.
We will see in Remark~\ref{rem:counterExamples} that this property, as well as its corollaries below, does not hold for subword complexes with double roots.

\begin{proposition}
\label{prop:intersection}
Let~$I$ and $J$ be two facets of a double root free subword complex~$\subwordComplex$.
Then the intersection~$I \cap J$ is contained in all facets of the interval~$[I,J]$ in the increasing flip graph~$\flipGraph$.
\end{proposition}

We extract the crucial part of the proof of this proposition in the following lemma.
\begin{lemma}
\label{lem:preparation}
Let~$I_0 \edge I_1 \edge \cdots \edge I_{\ell+1}$ be a path in the increasing flip graph~$\flipGraph$ with~$\positiveEdgeLabel_k \eqdef \positiveEdgeLabel(I_k \sep I_{k+1})$ and such that~$\positiveEdgeLabel_0 = \max\{ \positiveEdgeLabel_0,\ldots,\positiveEdgeLabel_\ell\}$. Then, starting from~$I_0$, it is possible to skip the first flip at position~$\positiveEdgeLabel_0$, and directly successively flip positions~$\positiveEdgeLabel_1, \positiveEdgeLabel_2, \dots, \positiveEdgeLabel_\ell$.
If~$I_0 = I'_1 \edge I'_2 \edge \cdots \edge I'_{\ell+1}$ is the corresponding path for which~$\positiveEdgeLabel(I'_k,I'_{k+1}) = \positiveEdgeLabel_k$ for all~$k \in [\ell]$, we moreover have that $\Root{I'_k}{p} = \Root{I_k}{p}$ for any position~$p \le \positiveEdgeLabel_0$ and any~$k \in [\ell+1]$.
\end{lemma}

\begin{proof}
The proof is based on the observation that flips are described using the root function, and that flipping out~$i$ and flipping in~$j$ only affects the roots located between positions~$i$ and~$j$, see Proposition~\ref{prop:roots&flips}.
Remember that a position $p$ in a facet $I$ is increasingly flippable if and only if the root $\Root{I}{p}$ is contained in the inversion set of $\rho$, compare Propositions~\ref{prop:roots&flips}\eqref{prop:roots&flips:flippable} and~\ref{prop:roots&flips}\eqref{prop:roots&flips:increasing}.

We prove the statement by induction on~$k$.
Namely, we prove that
\begin{enumerate}[(i)]
  \item $\Root{I'_1}{p} = \Root{I_1}{p}$ for all positions~$p \leq \positiveEdgeLabel_0$, and that \label{item:basecase}
  \item for any~$k \in [\ell]$, if $\Root{I'_k}{p} = \Root{I_k}{p}$ for all positions~$p \leq \positiveEdgeLabel_0$, then the position $\positiveEdgeLabel_k$ is increasingly flippable in $I'_k$ and $\Root{I'_{k+1}}{p} = \Root{I_{k+1}}{p}$ for all positions~${p \leq \positiveEdgeLabel_0}$. \label{item:inductioncase}
\end{enumerate}

To prove~\eqref{item:basecase}, observe that the flip $I_0 \edge I_1$ does not affect roots located to the left of position~$\positiveEdgeLabel_0$, so we have~${\Root{I'_1}{p} = \Root{I_0}{p} = \Root{I_1}{p}}$ for any position~$p \le \positiveEdgeLabel_0$.

To prove~\eqref{item:inductioncase}, we assume that $\Root{I'_k}{p} = \Root{I_k}{p}$ for all positions~$p \leq \positiveEdgeLabel_0$.
In particular,  $\Root{I'_k}{\positiveEdgeLabel_k} = \Root{I_k}{\positiveEdgeLabel_k}$ because $\positiveEdgeLabel_k \leq \positiveEdgeLabel_0$.
Since~$\positiveEdgeLabel_k$ is increasingly flippable in~$I_k$, this root is in the inversion set of~$\rho$, and therefore, $\positiveEdgeLabel_k$ is also increasingly flippable in~$I'_k$. Here, we used twice Propositions~\ref{prop:roots&flips}\eqref{prop:roots&flips:flippable} and~\ref{prop:roots&flips}\eqref{prop:roots&flips:increasing}.
Define now~$\negativeEdgeLabel_k \eqdef \negativeEdgeLabel(I_k \sep I_{k+1})$ and~$\negativeEdgeLabel'_k \eqdef \negativeEdgeLabel(I'_k \sep I'_{k+1})$.
If~$\negativeEdgeLabel_k \le \positiveEdgeLabel_0$, then
$${\Root{I'_k}{\positiveEdgeLabel_k} = \Root{I_k}{\positiveEdgeLabel_k} = \Root{I_k}{\negativeEdgeLabel_k} = \Root{I'_k}{\negativeEdgeLabel_k}},$$
and thus~$\negativeEdgeLabel_k = \negativeEdgeLabel'_k$. Here, we used twice Proposition~\ref{prop:roots&flips}\eqref{prop:roots&flips:flip}.
Similarly, if~$\negativeEdgeLabel'_k \le \positiveEdgeLabel_0$, then~$\negativeEdgeLabel_k = \negativeEdgeLabel'_k$.
We therefore obtain that either both~$\negativeEdgeLabel_k$ and~$\negativeEdgeLabel'_k$ are located to the right of~$\positiveEdgeLabel_0$, or~$\negativeEdgeLabel_k = \negativeEdgeLabel'_k$.
In both cases, we know that $\positiveEdgeLabel_k < p \le \negativeEdgeLabel_k$ if and only if~$\positiveEdgeLabel_k < p \le \negativeEdgeLabel'_k$ for any position~$p \leq \positiveEdgeLabel_0$.
Since~$\Root{I'_k}{p} = \Root{I_k}{p}$, we thus obtain that~$\Root{I'_{k+1}}{p} = \Root{I_{k+1}}{p}$ by Proposition~\ref{prop:roots&flips}\eqref{prop:roots&flips:update}.
\end{proof}

\begin{proof}[Proof of Proposition~\ref{prop:intersection}]
Let~$I = I_0 \edge I_1 \edge \cdots \edge I_\ell \edge I_{\ell+1} = J$ be a path from~$I$ to~$J$ in the increasing flip graph~$\flipGraph$.
For~$0 \le k \le \ell$, define $\positiveEdgeLabel_k \eqdef \positiveEdgeLabel(I_k \sep I_{k+1})$ and~$\negativeEdgeLabel_k \eqdef \negativeEdgeLabel(I_k \sep I_{k+1})$.
In other words,~$\positiveEdgeLabel_k \in I_k$, $\negativeEdgeLabel_k \in I_{k+1}$ and~${I_k \ssm \positiveEdgeLabel_k = I_{k+1} \ssm \negativeEdgeLabel_k}$.

We assume by means of contradiction that there is a position in $I \cap J$ flipped out during the flip path which is flipped back later in the path.
Up to shortening the path, we can assume without loss of generality that this position is flipped out during the first flip~$I_0 \edge I_1$ and flipped back in during the last flip~$I_\ell \edge I_{\ell+1}$, \ie $\positiveEdgeLabel_0 = \negativeEdgeLabel_\ell$.
We moreover assume that our path is a minimal length path which flips back in a position already flipped out.

Under these assumptions, we prove that
\begin{enumerate}[(i)]
\item $\positiveEdgeLabel_0 = \max \{ \positiveEdgeLabel_0, \dots, \positiveEdgeLabel_\ell \}$, \label{proof:intersection:i}
\item starting from facet~$I$, we can successively flip positions~$\positiveEdgeLabel_1, \positiveEdgeLabel_2, \dots, \positiveEdgeLabel_{\ell-1}$ (just skipping the first and the last flips at positions~$\positiveEdgeLabel_0$ and~$\positiveEdgeLabel_\ell$), and \label{proof:intersection:ii}
\item the facet~$J'$ obtained after these flips has a double root at positions~$\positiveEdgeLabel_0$ and~$\positiveEdgeLabel_\ell$. \label{proof:intersection:iii}
\end{enumerate}

To prove~\eqref{proof:intersection:i}, assume that the index~$m \in [0,\ell]$ such that~$\positiveEdgeLabel_m = \max\{\positiveEdgeLabel_0, \dots, \positiveEdgeLabel_\ell\}$ is different from~$0$. Note that~$0 < m < \ell$ since~$\positiveEdgeLabel_\ell < \negativeEdgeLabel_\ell = \positiveEdgeLabel_0$. Consider the path of flips
$$I = I_0 \edge \cdots \edge I_m = I'_m \edge I'_{m+1} \edge \cdots \edge I'_{\ell}$$
defined by $\positiveEdgeLabel(I_k, I_{k+1}) = \positiveEdgeLabel_k$ for~$k < m$ and $\positiveEdgeLabel(I'_k, I'_{k+1}) = \positiveEdgeLabel_{k+1}$ for~$k \ge m$.
In other words, starting from~$I$, we flip positions~$\positiveEdgeLabel_0, \dots, \positiveEdgeLabel_{m-1}, \positiveEdgeLabel_{m+1}, \dots, \positiveEdgeLabel_\ell$, skipping the flip at position~$\positiveEdgeLabel_m$.
According to Lemma~\ref{lem:preparation}, all flips in the path~$I'_m \edge I'_{m+1} \edge \cdots \edge I'_{\ell}$ are admissible since $\positiveEdgeLabel_k \le \positiveEdgeLabel_m$ for all~$k \ge m$, and we have
$$\Root{I'_{\ell-1}}{\positiveEdgeLabel_\ell} = \Root{I_{\ell}}{\positiveEdgeLabel_\ell} = \Root{I_{\ell}}{\positiveEdgeLabel_0} = \Root{I'_{\ell-1}}{\positiveEdgeLabel_0}.$$
Therefore, we flip back position~$\positiveEdgeLabel_0$ in facet~$I'_{\ell}$, thus contradicting the length minimality of the path~$I = I_0 \edge I_1 \edge \cdots \edge I_\ell \edge I_{\ell+1} = J$.

We now prove~\eqref{proof:intersection:ii} and~\eqref{proof:intersection:iii}. By~\eqref{proof:intersection:i}, the path~$I = I_0 \edge I_1 \edge \cdots \edge I_\ell \edge I_{\ell+1} = J$ satisfies the hypothesis of Lemma~\ref{lem:preparation}.
We therefore obtain directly~\eqref{proof:intersection:ii}.
Let~$J'$ denote the facet of~$\facets$ obtained after flipping successively~$\positiveEdgeLabel_1, \positiveEdgeLabel_2, \dots, \positiveEdgeLabel_{\ell-1}$ starting from~$I$.
We moreover obtain
$$\Root{J'}{\positiveEdgeLabel_\ell} = \Root{I_\ell}{\positiveEdgeLabel_\ell} = \Root{I_\ell}{\positiveEdgeLabel_0} = \Root{J'}{\positiveEdgeLabel_0},$$
where the first and last equalities are ensured by Lemma~\ref{lem:preparation}, while the middle one holds by Proposition~\ref{prop:roots&flips}\eqref{prop:roots&flips:flip} since we flip position~$\positiveEdgeLabel_\ell$ to position~$\negativeEdgeLabel_\ell = \positiveEdgeLabel_0$ in facet~$I_\ell$.
Since the facet~$J'$ contains both~$\positiveEdgeLabel_0$ and~$\positiveEdgeLabel_\ell$, it has a double root, thus proving~\eqref{proof:intersection:iii}.
\end{proof}

The following theorem is now a direct consequence of Proposition~\ref{prop:intersection}.

\begin{theorem}
\label{thm:fallingPath}
There is at most one $\positiveEdgeLabel$-falling (resp.~$\negativeEdgeLabel$-falling) path between any two facets~$I$ and~$J$ of a double root free subword complex~$\subwordComplex$.
If it exists, its length is given by~$|I \ssm J| = |J \ssm I|$.
\end{theorem}

\begin{proof}
Let~$I = I_1 \edge \cdots \edge I_{\ell+1} = J$ be a $\positiveEdgeLabel$-falling path from~$I$ to~$J$ in the increasing flip graph~$\flipGraph$, and define $\positiveEdgeLabel_k \eqdef \positiveEdgeLabel(I_k \sep I_{k+1})$ and~$\negativeEdgeLabel_k \eqdef \negativeEdgeLabel(I_k \sep I_{k+1})$.
For~$k < k'$, we then have~$\negativeEdgeLabel_k \ne \positiveEdgeLabel_{k'}$ (because the flips are increasing and the path is~$\positiveEdgeLabel$-falling) and~$\positiveEdgeLabel_k \ne \negativeEdgeLabel_{k'}$ (otherwise, the position~$\positiveEdgeLabel_k = \negativeEdgeLabel_{k'}$ would be flipped out and flipped back in during the path, thus contradicting Proposition~\ref{prop:intersection}).
This implies that~$\positiveEdgeLabel_k \in I \ssm J$ and~$\negativeEdgeLabel_k \in J \ssm I$ for all~$k \in [\ell]$.
Therefore~$\positiveEdgeLabel_k$ is the $k$\ordinal{} largest position of~$I \ssm J$ and~$\ell = |I \ssm J| = |J \ssm I|$.
This uniquely determines the $\positiveEdgeLabel$-falling path from~$I$ to~$J$.
The proof is similar for the $\negativeEdgeLabel$-falling path (see also Proposition~\ref{prop:bijectionFallingPaths}).
\end{proof}

\begin{corollary}
\label{coro:lengthPaths}
Let~$I$ and~$J$ be two facets of a double root free subword complex such that~$I \directedPath J$.
The unique $\positiveEdgeLabel$-rising (resp.~$\negativeEdgeLabel$-rising) path from~$I$ to~$J$ has maximal length among all path from~$I$ to~$J$.
Moreover, if there is a $\positiveEdgeLabel$-falling (resp.~$\negativeEdgeLabel$-falling) path from~$I$ to~$J$, it has minimal length.
\end{corollary}

\begin{proof}
Consider a maximal length path from~$I$ to~$J$.
According to the proof of Theorem~\ref{thm:ELlabeling}, we can modify this path to obtain the unique $\positiveEdgeLabel$-rising path from~$I$ to~$J$.
In the situation of a double root free subword complex, this procedure does not decrease the length of the path, since the first distinguished case in the proof of Lemma~\ref{lem:dihedral} cannot occur.
This proves the result for the $\positiveEdgeLabel$-rising path.
For the $\positiveEdgeLabel$-falling path, this follows directly from Theorem~\ref{thm:fallingPath}.
The proof is similar for the negative edge labeling~$\negativeEdgeLabel$.
\end{proof}

\begin{remark}
\label{rem:counterExamples}
Note that the conclusions of Proposition~\ref{prop:intersection}, Theorem~\ref{thm:fallingPath}, and Corollary~\ref{coro:lengthPaths} do indeed not hold if~$\subwordComplex$ has double roots.
Whenever one has a double root, one can reduce the situation to type $A_1$  with generator~$s$ for the word~$\Q = \sq{sss}$ and the element~$\rho = s$, using Proposition~\ref{prop:restriction} (one might actually get that the word $\Q$ contains more than three letters, but the argument stays the same).
In this case, the increasing flip graph~$\flipGraph$ consists of the two paths
$$
  \{ 1,2 \} \, \edgePositiveLabel{2} \, \{ 1,3 \} \, \edgePositiveLabel{1} \, \{ 2,3 \} \quad\text{and}\quad \{1,2\} \, \edgePositiveLabel{1} \, \{2,3\},
$$
where the numbers on the edges are their positive edge labels.
First, $\{ 1,3 \}$ lies in the interval~$[\{ 1,2 \}, \{ 2,3 \}]$ of the increasing flip graph~$\flipGraph$, but does not contain~$\{1,2\} \cap \{2,3\} = \{2\}$, thus contradicting Proposition~\ref{prop:intersection}.
Second, both paths are $\positiveEdgeLabel$-falling, contradicting the conclusions of Theorem~\ref{thm:fallingPath}.
Third, the second path is $\positiveEdgeLabel$-rising and shorter than the first $\positiveEdgeLabel$-falling path, contradicting the conclusions of Corollary~\ref{coro:lengthPaths}.
\end{remark}

\begin{corollary}
\label{coro:moebius}
The M\"obius function on the increasing flip poset~$\flipPoset$ of a double root free subword complex~$\subwordComplex$ is given by
$$ \mu(I,J) =
  \begin{cases}
    (-1)^{| J \ssm I|} & \text{ if there is a $\positiveEdgeLabel$-falling (resp.~$\negativeEdgeLabel$-falling) path from $I$ to $J$,} \\
    0                  & \text{ otherwise.}
  \end{cases}
$$
\end{corollary}

\begin{proof}
This is a direct consequence of Propositions~\ref{prop:Moebius} and~\ref{prop:bijectionFallingPaths} and Theorem~\ref{thm:fallingPath}.
\end{proof}

By this corollary, we can compute the M\"obius function of an interval~$[I,J]$ of the increasing flip poset as soon as we can decide whether or not there is a $\positiveEdgeLabel$-falling path from~$I$ to~$J$.
According to Proposition~\ref{prop:fallingPath}, there is always a $\positiveEdgeLabel$-falling path from the positive greedy facet to the negative greedy facet of a spherical subword complex.
We therefore obtain the value of the M\"obius function on the increasing flip poset~$\flipPoset$ of a spherical double root free subword complex.

\begin{corollary}
In a spherical double root free subword complex~$\subwordComplex$, we have
$$\mu\big(\positiveFacet,\negativeFacet\big) = (-1)^{|\Q| - \ell(\rho)}.$$
\end{corollary}

Observe again that this result fails if we drop the condition that~$\subwordComplex$ is spherical.
The subword complex~$\subwordComplex[\Qex,\rhoex]$ of Example~\ref{exm:toto} and the subword complex~$\subwordComplex[\sq{\tau}_1\sq{\tau}_2\sq{\tau}_1\sq{\tau}_2, \tau_1\tau_2]$ provide counter-examples.


\subsection{Two relevant examples}
\label{subsec:cambrian}

We finish this section by two relevant families of examples of double root free subword complexes, to which the above results can be applied.

\subsubsection{Cambrian lattices}
\label{subsubsec:cambrian}

We start with recalling background on sortable elements in Coxeter groups and Cambrian lattices.
Those were introduced by N.~Reading in~\cite{Reading-latticeCongruences, Reading-cambrianLattices, Reading-coxeterSortable, Reading-sortableElements}, originally to connect finite type cluster complexes to noncrossing partitions.
Fix a Coxeter element~$c$ of~$W$, and a reduced expression~$\sq{c}$ of~$c$.
That is to say, $\sq{c}$ is a word on~$S$ where each simple reflection appears precisely once.
For~$w \in W$, we denote by~$\sw{w}{c}$ the \defn{$\sq{c}$-sorting word} of~$w$, \ie the lexicographically first (as a sequence of positions) reduced subword of~$\sq{c}^\infty$ for $w$.
Moreover, this word can be written as~$\sw{w}{c} = \sq{c}_{K_1}\sq{c}_{K_2}\cdots\sq{c}_{K_p}$, where~$\sq{c}_{K}$ denotes the subword of~$\sq{c}$ only taking the simple reflections in~$K \subset S$ into account.
The element~$w$ is then called \defn{$c$-sortable} if ${K_1\supseteq K_2\supseteq\cdots\supseteq K_p}$.
Observe that the property of being $c$-sortable does not depend on the particular reduced expression~$\sq{c}$ of the Coxeter element~$c$.
We denote by~$\sortable{c}{W}$ the set of $c$-sortable elements in~$W$.
The order induced by the weak order on~$W$ turns~$\sortable{c}{W}$ into a lattice, the \defn{Cambrian lattice} for the Coxeter element~$c$~\cite{Reading-sortableElements}.

It was observed in~\cite[Remark~2.1]{Reading-coxeterSortable} that Cambrian lattices are naturally equipped with a search-tree structure.
The \defn{$\sq{c}$-sorting tree}~$\sortingTree$ has an edge between two $c$-sortable elements~$w$ and~$w'$ if the $\sq{c}$-sorting word for~$w$ is obtained from the one for~$w'$ by deleting the last letter.
See Example~\ref{exm:123} and \fref{fig:sortingTree}.
Observe that the $\sq{c}$-sorting tree really depends on the particular choice for the reduced expression~$\sq{c}$, and not only on the Coxeter element~$c$.

In their recent work~\cite{KallipolitiMuhle}, M.~Kallipoliti and H.~M\"uhle define an EL-labeling of the Cambrian lattice~$\sortable{c}{W}$ as follows.
They label a cover relation~$w \edge w'$ of~$\sortable{c}{W}$ by the first position within~$\sq{c}^\infty$ which is used in the $\sq{c}$-sorting word for~$w'$ but not in the $\sq{c}$-sorting word for~$w$.
They observed in~\cite[Remark~3.5]{KallipolitiMuhle} that the spanning tree formed by all rising paths from the source~$e$ to any other $c$-sortable element coincides with the $\sq{c}$-sorting tree mentioned above.
See Example~\ref{exm:123} and \fref{fig:sortingTree}.
They moreover use this EL-labeling to derive results on M\"obius functions of Cambrian lattices~\cite[Theorems~4.1,~4.2,~and~4.3]{KallipolitiMuhle}.

\begin{example}
\label{exm:123}
Let~$W = \fS_4$ and~$c = \tau_1\tau_2\tau_3$.
The $c$-sortable elements, the Hasse diagram of the Cambrian lattice, the EL-labeling of~\cite{KallipolitiMuhle}, and the $\sq{c}$-sorting tree are represented in \fref{fig:sortingTree}.
We write~$\it{12.1}$ instead of~$\tau_1\tau_2.\tau_1$ to simplify the picture (the dots mark the separation between the blocks~$\sq{c}_{K_i}$).

\begin{figure}[ht]
  \centerline{\includegraphics[width=\textwidth]{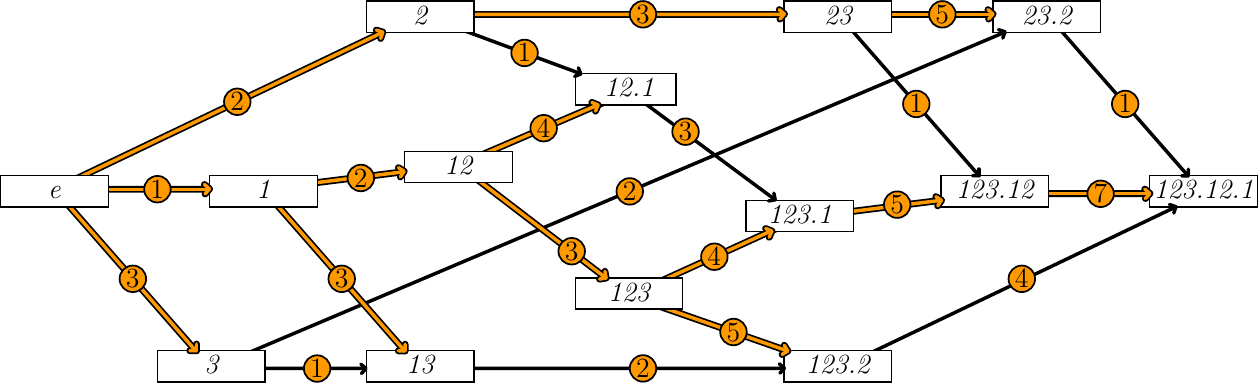}}
  \caption{The $(\sq{\tau}_1\sq{\tau}_2\sq{\tau}_3)$-sorting tree is the spanning tree corresponding to an EL-labeling of the $(\tau_1\tau_2\tau_3)$-Cambrian lattice.}
  \label{fig:sortingTree}
\end{figure}
\end{example}

We now recall that Cambrian lattices can be seen as increasing flip posets.
This interpretation was presented in~\cite[Sections~6.3.2 and~6.4]{PilaudStump}, based on previous connections between $c$-sortable elements and $c$-clusters~\cite{Reading-sortableElements}, and between $c$-clusters and facets of the subword complex~\cite{CeballosLabbeStump}.

Let~$\cwo{c}$ denote the $\sq{c}$-sorting word for the longest element~$w_\circ \in W$.
To simplify notations, we write~$\subwordComplex[\sq{c}]$ for the subword complex~$\subwordComplex[\sq{c}\cwo{c}, w_\circ]$.
Similarly, we denote by~$\facets[\sq{c}]$ its facets, by~$\flipGraph[\sq{c}]$ its increasing flip graph, by~$\flipPoset[\sq{c}]$ its increasing flip poset, and by~$\positiveSourceTree[\sq{c}]$ its positive source tree.
Following~\cite[Section~5.1]{PilaudStump}, we define a map~$\kappa : W \to \facets[\sq{c}]$ by sending an element~$w \in W$ to the unique facet~$\kappa(w)$ whose root configuration~$\Roots{\kappa(w)}$ is contained in~$w(\Phi^+)$.
For the subword complex~$\subwordComplex[\sq{c}]$, it turns out that the fibers of this map are intervals, and that their minimal elements are precisely the $c$-sortable elements.
This gives the following proposition.

\begin{proposition}[\protect{\cite[Corollary~6.31]{PilaudStump}}]
\label{prop:isomorphism}
The map associating to a facet~$I$ the unique (weak order) minimal element in~$\kappa^{-1}(I)$, is a poset isomorphism between the increasing flip poset and the Cambrian lattice.
\end{proposition}

Through this isomorphism, we can transfer the results discussed in this paper to Cambrian lattices.
We thus also obtain natural EL-labelings and spanning trees for Cambrian lattices.

\begin{example}
Let~$W = \fS_4$ and~$c = \tau_1\tau_2\tau_3$.
The facets of~$\subwordComplex[c]$, the Hasse diagram of~$\flipPoset[c]$, the positive edge labeling~$\positiveEdgeLabel$ of~$\flipGraph[c]$, and the positive source tree~$\positiveSourceTree[c]$ are represented in \fref{fig:greedyTreeCoxeter}.
Compare to \fref{fig:sortingTree}.

\begin{figure}[t]
  \centerline{\includegraphics[width=\textwidth]{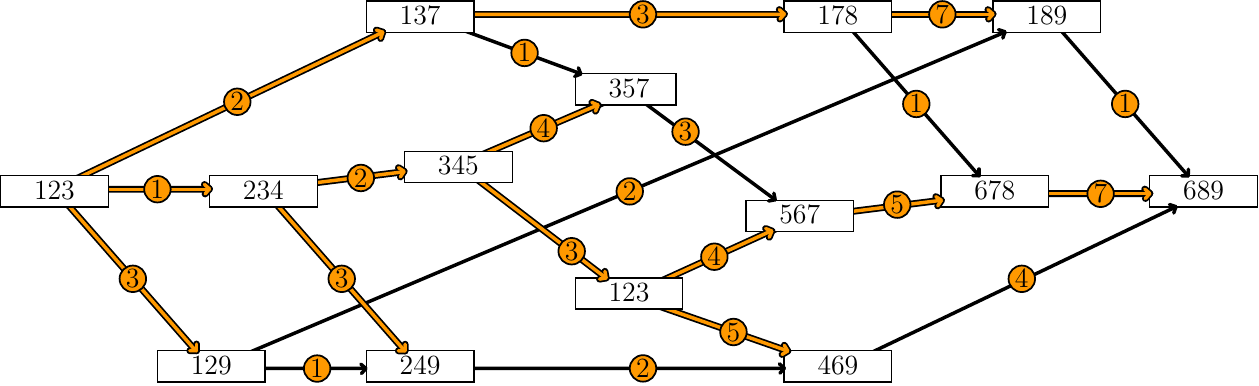}}
  \caption{The positive edge labeling~$\positiveEdgeLabel$ of~$\flipGraph[\sq{\tau}_1\sq{\tau}_2\sq{\tau}_3]$, and the positive source tree~$\positiveSourceTree[\sq{\tau}_1\sq{\tau}_2\sq{\tau}_3]$.}
  \label{fig:greedyTreeCoxeter}
\end{figure}
\end{example}

To finish, we want to observe that the positive edge labeling differs from the EL-labeling of~\cite{KallipolitiMuhle} and that the positive source tree~$\positiveSourceTree[\sq{c}]$ differs\footnote{The contrary was stated in a previous version of this paper. We thank an anonymous referee for pointing out this mistake.} from the $\sq{c}$-sorting tree~$\sortingTree$.
This is illustrated in the following (minimal) example.

\begin{example}[Positive source tree $\neq$ Coxeter-sorting tree]
Consider the Coxeter group~$W = \fS_{5}$ and the Coxeter element~$c = \tau_4 \tau_2 \tau_3 \tau_1$.
In this situation, the four facets of~$\subwordComplex[\sq{c}]$ given by
$$F_1 = \{1,8,9,11\},\quad F_2 = \{1,9,11,14\}, \quad F_3 = \{1,8,11,13\},\quad F_4 = \{1,11,13,14\},$$
are respectively sent by the isomorphism of Proposition~\ref{prop:isomorphism} to the $c$-sortable elements
$$w_1 = \tau_2\tau_3\tau_1.\tau_2, \quad w_2 = \tau_2\tau_3\tau_1.\tau_2\tau_1,\quad w_3 = \tau_2\tau_3\tau_1.\tau_2\tau_3, \quad w_4 = \tau_2\tau_3\tau_1.\tau_2\tau_3\tau_1.$$
The facets~$F_1, F_2, F_3, F_4$ (resp.~the $c$-sortable elements~$w_1, w_2, w_3, w_4$) form a square within the increasing flip poset (resp.~within the Cambrian lattice).
\fref{fig:differentTrees} represents the two EL-labelings and their corresponding spanning trees restricted to these squares.
The positive source tree~$\positiveSourceTree[\sq{c}]$ contains all edges of this square except~$F_3 \edge F_4$, while the $\sq{c}$-sorting tree~$\sortingTree$ contains all edges of this square except~$w_2 \edge w_4$.

\begin{figure}[t]
  \centerline{\includegraphics[width=\textwidth]{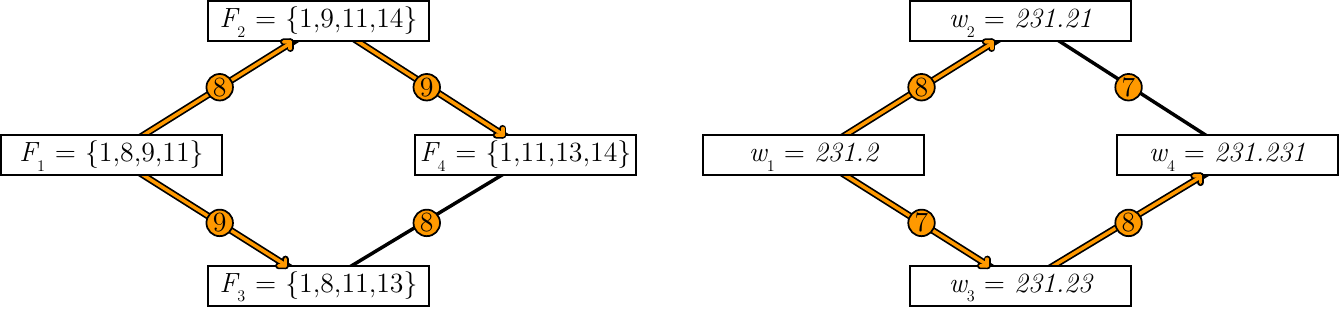}}
  \caption{The positive source tree~$\sortingTree$ differs from the $\sq{c}$-sorting tree.}
  \label{fig:differentTrees}
\end{figure}
\end{example}

\subsubsection{Duplicated words}
\label{subsubsec:duplicated}

Let~$\rho \eqdef \rho_1 \cdots \rho_\zeta$ be a reduced expression of an element~$\rho$ of~$W$.
For~$k \in [\zeta]$, we define a root~$\alpha_k \eqdef \rho_1 \cdots \rho_{k-1}(\alpha_{\rho_k})$.
Note that the roots $\alpha_1, \dots, \alpha_\zeta$ are pairwise distinct and positive.
They are the roots of the inversion set of~$\rho$.

Let~$X$ be an arbitrary subset of $\chi \eqdef |X|$ positions of~$[\zeta]$.
We denote by~$\Qdup$ the word on~$S$ with $\zeta + \chi$ letters which is obtained by duplicating the letters of~$\rho \eqdef \rho_1 \cdots \rho_\zeta$ at positions in~$X$.
To be more precise, define~${k^\bullet \eqdef k + |X \cap [k-1]|}$ for~${k \in [\zeta]}$.
Observe that $[\zeta+\chi] = \set{k^\bullet}{k \in [\zeta]} \sqcup \set{x^\bullet + 1}{x \in X}$.
Then, we set $\Qdup \eqdef \q_1 \cdots \q_{\zeta+\chi}$, where~$\q_{k^\bullet} \eqdef \rho_k$ for~$k \in [\zeta]$ and $\q_{x^\bullet+1} \eqdef \rho_x$ for~$x \in X$.
For~$k \in [\zeta]$, the position~$k^\bullet$ is the new position in~$\Qdup$ of the $k$\ordinal{} letter of~$\rho$, and for~$x \in X$, the position~$x^\bullet+1$ is the new position in~$\Qdup$ of the duplicated $x$\ordinal{} letter of~$\rho$.

For any~$x \in X$, the pair~$\{x^\bullet,x^\bullet+1\}$ of duplicated positions intersects any facet of~$\subwordComplex[\Qdup, \rho]$, otherwise the expression would not be reduced.
It follows that any facet of~$\subwordComplex[\Qdup, \rho]$ contains precisely one element of each pair~$\{x^\bullet,x^\bullet+1\}$ of duplicated positions and no other position.
Therefore, the facets of~$\subwordComplex[\Qdup, \rho]$ are precisely the sets~$I_\varepsilon \eqdef \set{x^\bullet + \varepsilon_x}{x \in X}$ where ${\varepsilon \eqdef (\varepsilon_1,\dots,\varepsilon_\chi) \in \{0,1\}^X}$.
Moreover, the roots of the facet~$I_\varepsilon$ of~$\subwordComplex[\Qdup, \rho]$ are given by~$\Root{I_\varepsilon}{k^\bullet} = \alpha_k$ for~${k \in [\zeta]}$ and~$\Root{I_\varepsilon}{x^\bullet+1} = (-1)^{\varepsilon_x}\alpha_x$ for~${x \in X}$.
Thus, the subword complex~$\subwordComplex[\Qdup, \rho]$ is double root free, since the roots~$\alpha_1, \dots, \alpha_\zeta$ are pairwise distinct.

The subword complex~$\subwordComplex[\Qdup, \rho]$ is the boundary complex of the $\chi$-dimensional cross polytope.
In particular, the graph of increasing flips~$\flipGraph[\Qdup, \rho]$ is the directed $1$-skeleton~$\square_\chi$ of a $\chi$-dimensional cube, and the increasing flip poset~$\flipPoset[\Qdup, \rho]$ is a boolean poset.

The positive greedy facet~$\positiveFacet[\Qdup, \rho]$ is the facet~$I_{\zero}$, while the negative greedy facet~$\negativeFacet[\Qdup, \rho]$ is the facet~$I_{\one}$.
The positive and negative edge labelings~$\positiveEdgeLabel$ and~$\negativeEdgeLabel$ of~$\subwordComplex[\Qdup, \rho]$ are essentially the same as the edge labeling~$\lambda$ of~$\square_\chi$ presented in Example~\ref{exm:cube}.
More precisely, for any edge~$\varepsilon \edge \varepsilon'$ of~$\square_\chi$, we have
$$\psi \circ \lambda(\varepsilon \sep \varepsilon') = \positiveEdgeLabel(I_\varepsilon \sep I_{\varepsilon'}) = \negativeEdgeLabel(I_\varepsilon \sep I_{\varepsilon'}) - 1,$$
where~$\psi : [\chi] \to \set{x^\bullet}{x \in X}$ is such that~$\psi(1) < \psi(2) < \dots < \psi(\chi)$.
Since ${\positiveEdgeLabel(\cdot) = \negativeEdgeLabel(\cdot)-1}$, the positive and negative source trees~$\positiveSourceTree[\Qdup, \rho]$ and~~$\negativeSourceTree[\Qdup, \rho]$ coincide.
Similarly the positive and negative sink trees~$\positiveSinkTree[\Qdup, \rho]$ and~$\negativeSinkTree[\Qdup, \rho]$ coincide as well.
Moreover, the map~$\varepsilon \mapsto I_\varepsilon$ defines a graph isomorphism from the $\lambda$-source tree of~$\square_\chi$ to the source trees~$\positiveSourceTree[\Qdup, \rho] = \negativeSinkTree[\Qdup, \rho]$, and from the $\lambda$-sink tree of~$\square_\chi$ to the sink trees~$\positiveSinkTree[\Qdup, \rho] = \negativeSinkTree[\Qdup, \rho]$.
See Example~\ref{exm:cube} and \fref{fig:cube}.

Finally, the M\"obius function on the increasing flip poset~$\flipPoset[\Qdup,\rho]$ is given by
$$\mu(I_\varepsilon, I_{\varepsilon'}) = \begin{cases} (-1)^{\delta(\varepsilon, \varepsilon')} & \text{if } \varepsilon \directedPath \varepsilon', \\ 0 & \text{otherwise,} \end{cases}$$
where~$\delta$ denotes the Hamming distance on the vertices of the cube.
See Example~\ref{exm:cube2}.


\section*{Acknowledgments}

We are very grateful to the two anonymous referees for their detailed reading of several versions of the manuscript, and for many valuable comments and suggestions, both on the content and on the presentation.
Their suggestions led us to the current version of Proposition~\ref{prop:intersection}, to correct a serious mistake in a previous version, and to improve several arguments in various proofs.

V.\,P.~thanks M.~Pocchiola for introducing him to the greedy flip algorithm on pseudotriangulations and for uncountable inspiring discussions on the subject.
We thank M.~Kallipoliti and H.~M\"uhle for mentioning our construction in~\cite{KallipolitiMuhle}.
Finally, we thank the Sage and Sage-Combinat development teams for making available this powerful mathematics software.

\bibliographystyle{alpha}
\bibliography{greedyFlip.bib}

\end{document}